\documentclass[a4paper,11pt,reqno]{amsart}

\usepackage{amsmath}
\usepackage{amsthm}
\usepackage{amsfonts}
\usepackage{amssymb}
\usepackage{enumerate}
\usepackage{url}

\newcommand\ext{{\rm{Ext}}}
\newcommand\cds{S}
\newcommand\Gzero{\widehat{G}}

\newcommand\conv{{\rm conv}}

\newcommand\cG{{\mathcal G}}

\newcommand\eps{\varepsilon}
\newcommand\Prb{\mathbb{P}}
\newcommand\N{\mathbb{N}}
\newcommand\Q{\mathbb{Q}}
\newcommand\Z{\mathbb{Z}}
\newcommand\R{\mathbb{R}}
\newcommand\spanof[1]{\left<#1\right>}
\newcommand\lid{{\ell_\infty^d}}

\renewcommand{\le}{\leqslant}
\renewcommand{\ge}{\geqslant}

\title{Random Geometric Graphs and Isometries of Normed Spaces}

\author[P.~Balister]{Paul Balister}
\address{Department of Mathematical Sciences, University of Memphis, Memphis TN 38152, USA.}
\email{pbalistr@memphis.edu}
\author[B.~Bollob\'as]{B\'ela Bollob\'as}
\address{Department of Pure Mathematics and Mathematical Statistics,
Wilberforce Road, Cambridge CB3\thinspace0WB, UK; {\em and\/}
Department of Mathematical Sciences, University of Memphis, Memphis TN 38152,
USA; {\em and\/} London Institute for Mathematical Sciences, 35a South St.,
Mayfair, London W1K\thinspace2XF, UK.}
\email{bollobas@dpmms.cam.ac.uk}
\thanks{The first and second author are partially supported by
NSF grant DMS~1301614 and MULTIPLEX no.\ 317532.}
\author[K.~Gunderson]{Karen Gunderson}
\address{Heilbronn Institute for Mathematical Research, School of
  Mathematics, University of Bristol, Bristol BS8 1TW, UK}
\email{karen.gunderson@bristol.ac.uk}
\author[I.~Leader]{Imre Leader}
\address{Department of Pure Mathematics and Mathematical Statistics,
Wilberforce Road, Cambridge CB3 0WB, UK.}
\email{I.Leader@dpmms.cam.ac.uk}
\author[M.~Walters]{Mark Walters}
\address{School of Mathematical Sciences, Queen Mary University of London, London E1 4NS, UK.}
\email{m.walters@qmul.ac.uk}


\begin{document}

\newtheorem{theorem}{Theorem}
\newtheorem{lemma}[theorem]{Lemma}
\newtheorem{proposition}[theorem]{Proposition}
\newtheorem{corollary}[theorem]{Corollary}
\newtheorem*{defn}{Definition}
\newtheorem*{theorem*}{Theorem}
\newtheorem*{lemma*}{Lemma}
\newtheorem*{conjecture}{Conjecture}
\newtheorem{question}{Question}
\newtheorem*{quotedresult}{Lemma A}
\theoremstyle{remark}
\newtheorem*{remark}{Remark}

\begin{abstract}
  Given a countable dense subset $S$ of a finite-dimensional normed
  space $X$, and $0<p<1$, we form a random graph on $S$ by joining,
  independently and with probability $p$, each pair of points at
  distance less than $1$. We say that $S$ is \emph{Rado} if any two such
  random graphs are (almost surely) isomorphic.

  Bonato and Janssen showed that in $\ell_\infty^d$ almost all $S$ are
  Rado. Our main aim in this paper is to show that $\ell_\infty^d$ is the
  unique normed space with this property: indeed, in every other space
  almost all sets $S$ are non-Rado. We also determine which spaces admit
  some Rado set: this turns out to be the spaces that have an
  $\ell_\infty$ direct summand. These results answer questions of Bonato and
  Janssen.

  A key role is played by the determination of which
  finite-dimensional normed spaces have the property that every
  bijective step-isometry (meaning that the integer part of distances
  is preserved) is in fact an isometry. This result may be of
  independent interest.
\end{abstract}
\maketitle
\section{Introduction}
In \cite{MR2854782} Bonato and Janssen introduced a new random
geometric graph model, defined as follows. Let $V$ be a
finite-dimensional normed space and let $\cds$ be a fixed countable dense
subset of $V$. Let $\Gzero=\Gzero(V,S)$ be the unit radius graph on
$\cds$: that is $x,y\in \cds$ are joined if $\|x-y\|<1$. Form
$G=G_p(V,S)$ by taking a random subgraph of $\Gzero(V,S)$ in which
each edge is chosen independently with probability~$p$, and let
$\cG_p(V,S)$ be the probability space of such graphs.

Motivated by the existence of the Rado graph, the unique infinite
graph in the Erd\H{o}s-R\'enyi random graph model, Bonato and Janssen
asked when the random graph in their model is almost surely unique up
to isomorphism.  We say a set $\cds$ is \emph{Rado} if the resulting
graph is almost surely unique up to isomorphism, and we say it is
\emph{strongly non-Rado} if any two such graphs are almost surely not
isomorphic. (Rather surprisingly, there are sets that are neither Rado
nor strongly non-Rado; see Theorem~\ref{t:precise-main-theorem} below.)

Bonato and Janssen proved that, for $V= \ell_\infty^d$ (the normed
space on $\R^d$ with norm defined by $\|(x_1,x_2,\dots, x_d)\|=\max_i
|x_i|$), almost all countable dense sets are Rado. [The exact
definition of `almost all' for countable dense sets is a little subtle
and we discuss it at the end of Section~\ref{s:notation}, but, for
now, we remark that the only property of `almost all' that we require
is that almost all sets contain no integer distances, and no integer
distances (or coincidences) in projections onto natural subspaces
such as the coordinate axes.]

In the same paper, Bonato and Janssen proved that all countable dense
sets in the Euclidean plane are strongly
non-Rado. Subsequently~\cite{MR3056370} they showed that almost all
countable dense sets in the plane with the hexagonal norm are strongly
non-Rado, and in~\cite{1310.4768B} that, for $\R^2$ with any norm that
is strictly convex or has a polygonal unit ball (apart from a
parallelogram), there are no Rado sets. They asked which normed spaces
contain a Rado set.

Our first result implies that $\ell_\infty^d$ is the only space for
which almost all countable dense sets are Rado.
\begin{theorem}\label{t:main-theorem}
  Let $V$ be a finite-dimensional normed space not isometric to
  $\ell_\infty^d$. Then, for any $0<p<1$, almost every countable dense
  set $\cds$ is strongly non-Rado.
\end{theorem}
\noindent%

Theorem~\ref{t:main-theorem} shows what happens for `typical'
countable dense sets~$\cds$, but leaves open the possibility of
exceptional cases. Our second result,
Theorem~\ref{t:precise-main-theorem} below, is a refinement of
Theorem~\ref{t:main-theorem} that answers the question of Bonato and
Janssen, and, in fact, describes the precise situation in each normed
space.

Before stating the theorem, we need the following fact about
finite-dimensional normed spaces, which roughly says that any such
space contains a unique maximal $\lid$ subspace embedded in an
$\ell_\infty$ fashion. The precise statement is that, for any
finite-dimensional normed space $V$, there exists a unique maximal
subspace $W$ isometric to $\lid$ for some $d$, such that there is a
subspace $U$ with $V=U\oplus W$ and $\|u+w\|=\max(\|u\|,\|w\|)$ for
all $u\in U$ and $w\in W$. We prove this result in
Section~\ref{s:l-inf-decomp}. This decomposition is useful since, in
essence, the complicated behaviour can only occur on the
$\ell_\infty^d$ part. We call this decomposition the
\emph{$\ell_\infty$-decomposition} and write it as $V=(U\oplus
\ell_\infty^d)_\infty$.

We are now ready to state the main result of the paper.

\begin{theorem}\label{t:precise-main-theorem}
  Let $V$ be a normed space with $\ell_\infty$-decomposition $(U\oplus
  \ell_\infty^d)_\infty$ as above, and let $0<p<1$. Then
  \begin{enumerate}[(i)]
    \item\label{e:i} If $V= \ell_\infty^d$ (i.e., $U={0}$ in the
    $\ell_\infty$-decomposition), then almost all countable dense
    sets~$\cds$ are Rado, but there exist countable dense sets which
    are strongly non-Rado.  Additionally, there exist countable dense
    sets $S$ for which the probability that two graphs $G,G'\in
    \cG_p(V,S)$ are isomorphic lies strictly between $0$ and $1$.
    \item\label{e:ii} If $d=0$ (i.e., $V=U$), then \emph{all}
    countable dense sets $\cds$ are strongly non-Rado.
    \item\label{e:iii} If $d>0$ and $U\not=\{0\}$ then, almost all
    countable dense sets $\cds$ are strongly non-Rado, but there exist
    countable dense sets $\cds$ which are Rado. Additionally, there
    exist countable dense sets for which the probability that two
    graphs $G,G'\in \cG_p(V,S)$ are isomorphic lies strictly between
    $0$ and $1$.
  \end{enumerate}
\end{theorem}
As we mentioned above, the typical case in~(\ref{e:i}) was proved by
Bonato and Janssen. In fact they proved more: they showed that the
graph is independent of $S$. More precisely, they showed that for
almost all countable dense sets $S$ and $S'$, and any $p,p'\in(0,1)$,
two graphs $G\in \cG_p(\ell_\infty^d,S)$ and $G'\in
\cG_{p'}(\ell_\infty^d,S')$ are almost surely isomorphic. Of course,
Theorem~\ref{t:precise-main-theorem} shows that this does not hold for
other normed spaces, as Parts~(\ref{e:ii}) and~(\ref{e:iii}) show that,
for almost all sets $S$, the probability that $G$ is isomorphic to any
particular graph is zero.

We shall make use of a key lemma of Bonato and Janssen that shows that
any graph isomorphism must induce an approximate isometric action on
$\cds$.
\begin{defn}
  Let $A\subseteq V$.  A \emph{step-isometry} on $A$ is a bijective
  function $f\colon A\to A$ such that, for all $x,y\in A$,
\[
 \left\lfloor \|x-y\|\right\rfloor=\left\lfloor \|f(x)-f(y)\|\right\rfloor.
\]
\end{defn}
We remark that Bonato and Janssen's definition was slightly different:
they did not require the function to be a bijection. However, all our
maps will be bijective, and many of the results we state only hold for
bijective step-isometries, so we use the above definition. Note that
we use `isometry' to mean any distance preserving map; in particular, it
need not be surjective.

Bonato and Janssen~\cite{MR2854782} proved the following
lemma.
\begin{lemma}[Bonato and Janssen~\cite{MR2854782}]\label{l:bonato-janssen}
  Suppose $G\in\cG_p(V,S)$. Then, almost surely, for every pair of
  points $x,y\in S$ and every $k\in \N$ with $k\ge 2$ we have
  $\|x-y\|<k$ if and only if $d_G(x,y)\le k$.

  In particular, for almost all graphs $G,G'$ in $\cG_p(V,S)$, every
  function $f\colon \cds\to \cds$ inducing an isomorphism of the
  graphs is a step-isometry on $\cds$.
\end{lemma}
\noindent%
To see why this is true, first note that it is immediate that the
existence of a path of length $k$ implies that the norm distance is
less than $k$.  For the converse, they use the countable dense
property to construct infinitely many disjoint paths of length $k$
between $x$ and $y$ in $\Gzero$. Each of these has a positive chance
of occurring in $G$ so, almost surely, one of them does.

The second part now follows since an isomorphism between any two
graphs satisfying the first part must be a step-isometry. (The case of
$\|x-y\|<1$ requires a small additional check.)

This result shows that a natural step towards characterising the
possible graph isomorphisms is to characterise all the step-isometries
and, indeed, this will form the bulk of this paper. As we shall prove,
any step-isometry of $\cds$ extends to a step-isometry of $V$
itself. Thus, we want to characterise the step-isometries of $V$.

Observe that a step-isometry on $V$ need not be an isometry. Indeed,
consider the following example on $\R$. Let $g\colon[0,1)\to[0,1)$ be
any increasing bijection. Now define $f(x)=\lfloor x\rfloor
+g(x-\lfloor x\rfloor)$. It is easy to see that this is a
step-isometry but not an isometry (unless $g$ is the identity
function).

This example extends naturally to $\ell_\infty^d$: we can do the above
independently in each coordinate.  However, the following result
shows this is essentially the only example.

We need one piece of notation first. If $V=U\oplus W$ is a vector
space and $f\colon V\to V$, then we say $f$ \emph{factorises over the
  decomposition} if there exist $f_U\colon U\to U$ and $f_W\colon W\to
W$ such that $f(u+w)=f_U(u)+f_W(w)$ for all $u\in U$ and $w\in W$. We
write $f=f_U\oplus f_W$.
\begin{theorem}\label{t:step-isometry}
  Let $V$ be a finite-dimensional normed space with
  $\ell_\infty$-decomposition $V= (U\oplus \ell_\infty^d)_\infty$, and
  let $f\colon V\to V$ be a step-isometry. Then $f$ factorises over
  the decomposition as $f=f_U\oplus f_{ \ell_\infty^d}$, where $f_U$
  is a bijective isometry of $U$ and $f_{ \ell_\infty^d}$ is a
  step-isometry of $\ell_\infty^d$.
\end{theorem}
Thus, to obtain a full characterisation of the step-isometries of $V$,
we need to classify the step-isometries of $\ell_\infty^d$. The
following result does exactly that.
\begin{theorem}\label{t:l_infinity}
  Let $f$ be a step-isometry of $\ell_\infty^d$. Then there exists a
  permutation $\sigma$ of $[d]$, and
  $\eps=(\eps_1,\eps_2,\dots,\eps_d)\in\{-1,+1\}^d$, and, for each $i$,
  an increasing bijection $g_i\colon [0,1)\to [0,1)$, such that
  \[
  f\left(\sum_{i=1}^d\lambda_i e_i\right)-f(0)=
  \sum_{i=1}^d\left(g_i\left(\lambda_i-\lfloor\lambda_i\rfloor\right)
    +\lfloor\lambda_i\rfloor\right)\eps_ie_{\sigma(i)}.
\]
where $e_1,e_2,\dots,e_d$ is the standard basis of $\lid$.
\end{theorem}
Having established these two theorems, as we shall see,
it is relatively straightforward to prove
Theorems~\ref{t:main-theorem} and~\ref{t:precise-main-theorem}.

The layout of this paper is as follows. In the next section we
introduce some standard definitions and notation, and then in
Section~\ref{s:l-inf-decomp} we prove the existence and uniqueness of
the $\ell_\infty$-decomposition together with some simple facts about
it that will be useful later.  In Section~\ref{s:extend} we prove that
any step-isometry on a dense subset can be extended to a step-isometry
on the whole space.

In Sections~\ref{s:extreme}-\ref{s:l-inf-isom} we prove
Theorems~\ref{t:step-isometry} and~\ref{t:l_infinity}.  The proofs of
these are quite lengthy, and we break them down as
follows. Sections~\ref{s:extreme} and~\ref{s:lattice} show that any
step-isometry is an isometry on the set of finite sums of extreme
points of the unit ball of $V$, and that we can compose the
step-isometry with an isometry so that the combination fixes all these
finite sums.  Then, Sections~\ref{s:extreme-lines}
and~\ref{s:fixed-subspace} show that any step isometry that fixes
these finite sums actually preserves many directions, and that this
implies it must fix a particular subspace. Finally,
Section~\ref{s:complement=well-spanned} shows that this particular
subspace is the non-$\ell_\infty$-component of the
$\ell_\infty$-decomposition, and Sections~\ref{s:step-isom}
and~\ref{s:l-inf-isom} put these facts together to complete the proofs
of Theorems~\ref{t:step-isometry} and~\ref{t:l_infinity}.

Parts of the proof of Theorem~\ref{t:precise-main-theorem} rely on the
back and forth method; as we use this several times we abstract it out
into Section~\ref{s:bf}. Then, in Section~\ref{s:main-theorem}, we use
Theorem~\ref{t:step-isometry} to prove
Theorem~\ref{t:precise-main-theorem}. We conclude with a brief
discussion of some other exceptional cases and some open problems.

\section{Normed Space Preliminaries}\label{s:notation} 
Throughout this paper we will be working exclusively in
finite-dimensional normed spaces, and we shall frequently make use of
properties particular to such spaces, such as the compactness of the
unit ball, and the fact that a linear injection from the space to
itself is necessarily a bijection.

Before stating any of the results that we need, we introduce some very
basic notation. Given a normed space $V$, we write $B(x,r)$ for the
closed ball of radius $r$ about $x$ and, on the few occasions we need
it, $B^\circ(x,r)$ for the open ball. 

In many cases the normed space will decompose naturally into
subspaces, $V=U\oplus W$. Given a vector $v=u+w$, with $u\in U$ and
$w\in W$, we call $u$ the \emph{$U$-component of $v$}. In most cases we
use the `additive' notation $u+w$ for vectors, and $V=U\oplus W$ for
subspaces. However, in some cases it will be easier to think of a
vector $v\in V$ as the ordered pair $(u,w)$ and the space as
$V=U\times W$, and we will occasionally use this
alternative notation.

Much of our work will be on (not necessarily linear) functions mapping
the vector space $V$ to itself. One key tool that we shall use several
times is the Mazur-Ulam Theorem (see,
e.g.,~\cite{fleming2002isometries}). This states that
any isometry is `affine'; i.e., a translation of a linear map. More
formally:
\begin{theorem}[Mazur-Ulam Theorem]\label{t:mazur-ulam}
  Let $X$ and $Y$ be normed spaces and $f\colon X\to Y$ be a
  surjective isometry. Then the map $\hat f \colon X\to Y$ given by
  $\hat f(x)=f(x)-f(0)$ is linear.
\end{theorem}

Since we are concerned only with finite-dimensional normed spaces in
this paper, it is worth noting that the Mazur-Ulam Theorem has a
particularly simple form in this setting.
\begin{corollary}\label{c:fin-dim-mazur-ulam}
  Suppose $V$ is a finite-dimensional normed space and that $f\colon
  V\to V$ is an isometry. Then $f$ is an affine bijection.
\end{corollary}
\begin{proof}
  By the Mazur-Ulam Theorem it suffices to show that $f$ is
  surjective.  First, observe that, by translating $f$ if necessary,
  we way assume that $f(0)=0$.

  We claim that $f(V)$ is closed. Indeed, if a sequence $f(x_n)$ tends
  to $y$, then $f(x_n)$ is Cauchy. This implies that, since $f$ is an
  isometry, the sequence $(x_n)$ is Cauchy, and thus converges to some point, $x$
  say. But then $f(x)=y$, which completes the proof of the claim.

  Now, suppose, for a contradiction, that there is some point $x\not
  \in f(V)$.  By the claim, $f(V)$ is closed, so there exists $\eps>0$
  such that the open ball $B^\circ(x,\eps)$ is disjoint from $f(V)$.

  Trivially, this implies that, for any $n\ge 1$, $f^n(x)\not \in
  B^\circ(x,\eps)$ or, equivalently, that $\|f^n(x) -x\|>\eps$. Since
  $f$ is an isometry, this shows that, for any $n>m\ge 1$, we have
  $\|f^n(x)-f^m(x)\|=\|f^{n-m}(x)-x\|>\eps$; i.e., the
  sequence $x,f(x),f^2(x)\dots$ is $\eps$-separated. But these terms all
  have norm $\|x\|$ (since $f(0)=0$), so this contradicts the
  compactness of the closed ball $B(0,\|x\|)$.
\end{proof}
Much of our work will concern properties of the closed unit ball
$B=B(0,1)$, and we recall some simple facts and notation related to
$B$.

The ball $B$ is a convex compact set, and the norm is determined by
$B$.  An extreme point of $B$ is a point $x$ such that if $y,z\in B$
and $x$ is a convex combination of $y,z$ then $y=z=x$. We write
$\ext(B)$ for the extreme points of $B$. The set $B$ is the convex
hull of its extreme points; i.e., $\conv(\ext(B))=B$. Since $B$ is not
contained in any proper subspace we see that the vectors in $\ext(B)$
span all of $V$. For any set of vectors $A$ we use $\spanof{A}$ to
denote the span of the vectors in $A$.

It will be useful to work with finite sums of extreme points. Thus, we
let $\Lambda$ be the `lattice' generated by the extreme points of the
unit ball $B$: that is all points of the form $\sum_i\lambda_i x_i$
with $\lambda_i\in \Z$ and $x_i\in \ext(B)$. Note that $\Lambda$ need
not be discrete. 

We start with a simple lemma that shows that $\Lambda$ is not too
sparse.
\begin{lemma}\label{l:lambda-cover}
  Let $V$ be a finite-dimensional normed space and let $v\in V$. Then
  there exists $x\in \Lambda$ such that $\|x-v\|\le \dim V/2$.
\end{lemma}
\begin{proof}
  As noted above, the extreme points of $B$ span $V$, so let
  $x_1,x_2,\dots x_d$, where $d=\dim V$, be any minimal spanning set
  of extreme points of $B$. Note that $\|x_i\|=1$ for all $i$.

  We can write $v=\sum_{i=1}^d a_ix_i$. For each $i$ let $\lambda_i$ be $a_i$
  rounded to the nearest integer (i.e., $\lambda_i=\lfloor
  a_i+1/2\rfloor$.)
  Then 
  \[\|v-x\|=\left\|\sum_{i=1}^d(a_i-\lambda_i)x_i\right\|\le
  \sum_{i=1}^d|a_i-\lambda_i|\|x_i\|\le d/2\] 
  as claimed.
\end{proof}
We remark that it is easy to see that this bound is obtained for the
space~$l_1^d$ (i.e., $\R^d$ with norm
$\|(x_1,x_2,\dots,x_d)\|=\sum_{i=1}^d|x_i|$).

Since the set $\Lambda$ need not be discrete we will often work with
its closure $\overline{\Lambda}$ which has a relatively simple form.
\begin{lemma}\label{l:basis-isom}
  Let $V$ be a $d$-dimensional normed space. Then
  there is a basis $e_1,e_2,\dots,e_d$ of unit vectors in $V$ and an $r\le d$ such
  that
  \[
  \overline\Lambda=\sum_{i<r}\R e_i\oplus\sum_{i\ge r}\Z e_i.
  \]
\end{lemma}
\begin{proof}
  As remarked above the extreme points of $B$ span $V$, so $\Lambda$
  spans $V$.  Hence $\overline\Lambda$ is a closed additive subgroup
  of $V\equiv \R^d$ so must have the form specified (see,
  e.g.,~\cite{GNANT}).
\end{proof}
The following subspace will be important later.
\begin{defn}
  We call the subspace $\sum_{i<r}\R e_i$ in the decomposition given
  by Lemma~\ref{l:basis-isom} the \emph{continuous subspace} of
  $\overline\Lambda$ and we usually denote it $U_0$.
\end{defn}
We make the following simple observation for future reference.
\begin{corollary}\label{c:coset-cover}
  The extreme points of the unit ball $B$ are covered by finitely many
  cosets of the continuous subspace $U_0$.\qed
\end{corollary}

We conclude this section with a brief discussion of the meaning of
`almost all' for countable dense sets. Before doing this we remark
that, for our purposes, all we need is the following: if
$V=(U\oplus\lid)_\infty$ then, for almost all sets $S$, no two points
of $S$ have the same $U$-component, nor differ by an integer in any
coordinate direction in their $\lid$-component. This obviously holds
for any sensible definition of `almost all.'

Indeed, there are several possible definitions in the literature, any
of which would be suitable. One such possibility is to take any
distribution on $\R^d$ with a strictly positive density function, and
let $S$ be the set formed by taking countably many independent samples
from it. Another would be to take the union of countably many
density-one Poisson Processes. (There are also rather less intuitive
possibilities -- for example taking $S$ to be the set of all local
minima of a Brownian motion on $\R^d$ -- see, e.g.,~\cite{MR2217814}
for a more complete discussion.)

\section{The $\ell_\infty$-decomposition}\label{s:l-inf-decomp}
In this section we prove the existence of the
$\ell_\infty$-decomposition mentioned in the introduction.
\begin{defn}
  A unit vector $v$ is an \emph{$\ell_\infty$-direction} if there
  exists a subspace $U$ of $V$ such that $V=(\spanof{v}\oplus
  U)_\infty$; i.e., $\|\alpha v+u\|=\max(|\alpha|,\|u\|)$ for all
  $\alpha\in \R$ and $u\in U$. We call $U$ the subspace
  corresponding to $v$. Note, we view
  $v$ and $-v$ as the \emph{same} $\ell_\infty$-direction.
\end{defn}
This definition is useful since, in any decomposition of $V$ as
$(U\oplus\ell_\infty^d)_\infty$ then each basis vector of the
$\ell_\infty^d$ is an $\ell_\infty$-direction; see
Proposition~\ref{p:l-inf-decomposition} for a formal proof.

\begin{lemma}
  Suppose that $v$ is an $\ell_\infty$-direction. Then the  corresponding
  subspace $U$ is unique.
\end{lemma}
\begin{proof}
  Fix a corresponding subspace $U$.  Suppose $u'\in V$ is any vector
  satisfying $\|\alpha v+ u'\|=\max(|\alpha|,\|u'\|)$. We can write
  $u'=\beta v+u$ for some $\beta\in\R$ and $u\in U$. By the definition
  of an $\ell_\infty$-direction, $\|u'\|\ge \|u\|$. Let
  $\gamma=\|u'\|$. By our assumption on $u'$ we have $\|u'+\gamma
  v\|=\|u'-\gamma v\|$ so
  $\|u+(\beta+\gamma)v\|=\|u+(\beta-\gamma)v\|$. Since $\gamma\ge
  \|u\|$ this implies $\beta=0$; i.e., $u'\in U$.
\end{proof}

\begin{lemma}\label{l:l-inf-orthogonal}
  Suppose that $v_1$ and $v_2$ are distinct $\ell_\infty$-directions
  with corresponding subspaces $U_1$ and $U_2$. Then $v_2\in U_1$.
\end{lemma}
\begin{proof}
  First, we claim that, for any vector $v'$, the line $\{v'+\lambda
  v_2:\lambda\in \R\}$ either contains a non-trivial interval of
  vectors of minimal norm (among points on the line), or
  contains~0. Indeed, this line contains a point, say $u'$ of
  $U_2$. Thus, we can write the line as $\{u'+\lambda v_2:\lambda\in
  \R\}$. Since $\|u'+\lambda v_2\|=\max(|\lambda|,\|u'\|)$ we see
  that, if $u'=0$, we have the latter case; and if $\|u'\|>0$ all
  vectors in the set $\{u'+\lambda v_2:|\lambda|\le \|u'\|\}$ have
  minimal norm. The claim follows.

  We can write $v_2=\alpha v_1+\beta u_1$
  with $u_1\in U_1$ and $\|u_1\|=1$. If $\alpha=0$ then $v_2$ is in
  $U_1$ as claimed; if $\beta=0$ then $v_2=\pm v_1$ so $v_2$ is the
  same $\ell_\infty$-direction as $v_1$ contradicting the assumption that
  $v_1$ and $v_2$ are distinct $\ell_\infty$-directions.

  Thus, we assume $\alpha,\beta\not =0$ and, by negating either or
  both of $v_1$ and $u_1$ we may assume $\alpha,\beta>0$.  Consider
  the set of vectors
  \[\{v_1-u_1+\lambda v_2:\lambda\in \R\}.\] Since
  \[\|v_1-u_1+\lambda v_2\|=\|(1+\lambda\alpha)v_1-(1-\lambda
  \beta)u_1\|=\max(|1+\lambda\alpha|,|1-\lambda\beta\|),\] we see that
  $\lambda=0$ gives the unique vector of minimal norm in this set, and
  that this vector has norm one which contradicts the above claim
  that, whenever the minimum norm on the line is not zero, there must
  be an interval of minimal norm.
  \end{proof}
  The next lemma shows that any set of $\ell_\infty$-directions combine
  to give an  $\ell_\infty$ subspace of $V$.

  \begin{lemma}\label{l:l-inf-components}
    Suppose that $v_1,v_2,\dots,v_k$ are any (distinct)
    $\ell_\infty$-directions with corresponding subspaces
    $U_1,U_2,\dots,U_k$. Then
    \[
    V=\left(\spanof{v_1}\oplus\spanof{v_2}\oplus\dots\oplus \spanof{v_k}
      \oplus \bigcap_{i=1}^k U_i \right)_\infty.
    \]
  \end{lemma}
  \begin{proof}
    First we show inductively that we can write any vector $v$ as
    $\sum_{i=1}^j \lambda_iv_i+ w_j$ where
    $w_j\in\bigcap_{i=1}^jU_i$. For $j=1$ it is just the definition of
    an $\ell_\infty$-direction. Suppose it holds for $j$. Then since
    $v_{j+1}$ is an $\ell_\infty$-direction we can write
    $w_j=\lambda_{j+1}v_{j+1}+w_{j+1}$ for some $w_{j+1}\in
    U_{j+1}$. Since, for each $1\le i\le j$, $w_j\in U_i$ and
    $v_{j+1}\in U_i$ we see that $w_{j+1}\in U_i$. Hence $w_{j+1}\in
    \bigcap_{i=1}^{j+1}U_i$ and the induction is complete.

    Next we show that the sum 
    \[
    \spanof{v_1}\oplus\spanof{v_2}\oplus\dots\oplus \spanof{v_k}
    \oplus \bigcap_{i=1}^k U_i 
    \]
    is direct.  Suppose that $u\in \bigcap_{i=1}^k U_i$, that
    $u+\sum_{i=1}^k\lambda_iv_i=0$ is a non-trivial linear relation,
    and that  $\lambda_j\not=0$. By Lemma~\ref{l:l-inf-orthogonal},
    $v_i\in U_j$ for all $i\not=j$, and obviously $u\in U_j$. Hence
    $v_j=\frac{1}{\lambda_j}\left(-u-\sum_{i\not=j}\lambda_iv_i\right)\in
    U_j$ which is a contradiction.

    To complete the proof observe that, by applying the
    $\ell_\infty$-direction property inductively, we have
    \[
    \left\|\sum_{i=1}^j \lambda_iv_i+ u\right\| = \max
    (|\lambda_1|,|\lambda_2|,\dots,|\lambda_j|,\|u\|)
    \]
  for any $j$, $\lambda_i\in\R$ and $u\in \bigcap_{i=1}^jU_i$. Taking $j=k$
  gives the result.
  \end{proof}
  Thus we see that the $\ell_\infty$-decomposition is unique in the
  strongest possible sense: namely that the $\ell_\infty^d$-component
  is the space spanned by \emph{all} the $\ell_\infty$-directions. We
  sum this up in the following proposition.
  \begin{proposition}\label{p:l-inf-decomposition}
    Suppose $V$ is a finite-dimensional normed space. Then there is a
    unique maximal space $W$ isometric to $\ell_\infty^d$, for some $d$,
    with the property that there is a subspace $U$ with $V=U\oplus W$
    and $\|u+w\|=\max(\|u\|,\|w\|)$, for any $u\in U$ and $w\in W$.

    Moreover, if $v_1,v_2,\dots,v_d$ are all the
    $\ell_\infty$-directions with corresponding subspaces
    $U_1,U_2,\dots,U_d$ then $W=\spanof{v_1,v_2,\dots,v_d}$ and
    $U=\bigcap_{i=1}^d U_i$.
  \end{proposition}
  \begin{proof}
    As in the statement of the proposition let $v_1,v_2,\dots,v_d$ be
    all the $\ell_\infty$-directions, $W=\spanof{v_1,v_2,\dots,v_d}$,
    and $U=\bigcap_{i=1}^dU_i$ where $U_i$ is the corresponding
    subspace to $v_i$. By Lemma~\ref{l:l-inf-components} $V=U\oplus
    W$, and for any $u\in U$ and $w=\sum_{i=1}^d\lambda_iv_i\in W$ we
    have $\|w\|=\max(|\lambda_1|,|\lambda_2|,\dots,|\lambda_d|)$ so
    $W$ is isometric to $\lid$ and, by By
    Lemma~\ref{l:l-inf-components} again,
    \[
    \|u+w\|=\max(\|u\|,|\lambda_1|,|\lambda_2|,\dots,|\lambda_d|)=\max(\|u\|,\|w\|)
    \]
    as required.
    
    To complete the proof suppose that $W'$ is any subspace isometric
    to $\ell_\infty^{d'}$ for some ${d'}$ and that $U'$ is a subspace with
    the property that $V=U'\oplus W'$ and
    $\|u'+w'\|=\max(\|u'\|,\|w'\|)$ for any $u'\in U'$ and $w'\in
    W'$. Let $e_1,e_2,\dots,e_{d'}$ be the natural basis of $W'$ viewed
    as $\ell_\infty^{d'}$. We see that, for any
    $\lambda_1,\lambda_2,\dots,\lambda_{d'}$ and any $u'\in U'$,
    \begin{align*}
    \|u'+\sum_{i=1}^{d'}\lambda_ie_i\|&=\max\left(\|u'\|,\|\sum_{i=1}^{d'}\lambda_ie_i\|\right)\\
    &=\max\left(\|u'\|,|\lambda_1|,|\lambda_2|,\dots,|\lambda_{d'}|\right)\\
    &=\max \left(|\lambda_1|,\|u+\sum_{i=2}^{d'}\lambda_ie_i\|\right),
  \end{align*}
  so, in particular, $e_1$ is an $\ell_\infty$-direction with
  corresponding subspace $U'\oplus\spanof{e_2,e_3,\dots,e_{d'}}$. Thus
  $e_1$ is one of the $v_i$ or $-v_i$ and, in particular, $e_1\in
  W$. Since this is true for each $e_i$, $1\le i\le {d'}$, we see that
  $W'\subseteq W$.
\end{proof}

\begin{corollary}\label{c:isom-decomp}
  Let $Q$ be a linear isometry of a finite-dimensional normed spaced
  $V$ with $\ell_\infty$-decomposition $U\oplus \ell_\infty^d$. Then
  $Q$ factorises over the decomposition as $Q_U\oplus
  Q_{\ell_\infty^d}$ and each factor is an isometry.
\end{corollary}
We remark that there are direct proofs of this result, based on
Proposition~\ref{p:l-inf-decomposition}; our proof, whilst a little
longer, will be useful for the next result
\begin{proof}
  First, observe that, since $Q$ is linear, factorising over the
  decomposition is the same as saying $Q(U)\subseteq U$ and
  $Q(\ell_\infty^d)\subseteq \ell_\infty^d$, and this is what we shall
  show.

  Suppose $v_1,v_2,\dots,v_d$ are the $\ell_\infty$-directions with
  corresponding subspaces $U_1,U_2,\dots U_d$. Let $v_i'=Q(v_i)$ and
  $U_i'=Q(U_i)$ for each $i$. We claim that $v_i'$ is an
  $\ell_\infty$-direction with subspace $U_i'$. Indeed, given $v'\in
  V$ let $v=Q^{-1}(v')$. Since $v_i$ is an $\ell_\infty$-direction we
  can write $v=\alpha v_i+u_i$ for some $u_i\in U_i$ and we have
  $\|v\|=\max(|\alpha|,\|u_i\|)$. Since $Q$ is linear, and writing
  $u_i'$ for $Q(u_i)$, this implies that $v'=Q(v)=Q(\alpha
  v_i+u_i)=\alpha v_i'+u_i'$ with $u_i'\in U_i'$. Since $Q$ is an
  isometry we have
  \[
  \|v'\|=\|v\|=\max(|\alpha|,\|u_i\|)=\max(|\alpha|,\|u_i'\|)
  \]
  as claimed.

  Thus $Q$ permutes the $\ell_\infty$-directions (possibly negating
  some of them) and, in particular, maps
  $\ell_\infty^d=\spanof{v_1,v_2,\dots,v_d}$ to itself. Also, $Q$
  permutes the corresponding subspaces so $U=\bigcap_{i=1}^dU_i$ is
  also mapped to itself. As observed above this shows that $Q$
  factorises as $Q|_U\oplus Q|_{\ell_\infty^d}$ and, since the factors
  are just the restrictions of $Q$ to $U$ and $\ell_\infty^d$
  respectively, we see that each factor is an isometry.
  \end{proof}
  The proof of Corollary~\ref{c:isom-decomp} actually describes what
  the isometries of $\lid$ are.
\begin{corollary}\label{c:l-inf-isom}
  Suppose that $f$ is an (bijective) isometry of $\ell_\infty^d$. Then
  there is a permutation $\sigma$ of $[d]$ and
  $\eps=(\eps_1,\eps_2,\dots,\eps_d)\in\{-1,+1\}^d$ such that $f$ is
  the linear map that sends each basis vector $e_i$ to $\eps_i
  e_{\sigma(i)}$, combined with a translation.
\end{corollary}
\begin{proof}
  Define $\hat f$ by $\hat f(x)=f(x)-f(0)$.  By the Mazur-Ulam theorem
  $\hat f$ is linear.  The proof of Corollary~\ref{c:isom-decomp}
  shows that $\hat f$ permutes the basis vectors of $\lid$ (which are
  obviously the $\ell_\infty$-directions), possibly changing the
  sign. The result follows.
\end{proof}

\section{Extending Step-Isometries from $\cds$ to $V$}\label{s:extend}
Suppose that $f$ is a step isometry on a dense set $S$ in $V$.  In
this section we show that $f$ extends to a continuous step-isometry
$\bar f\colon V\to V$.

As one would expect we shall define $\bar f$ in terms of sequences in
$S$. We start by proving some simple results about such sequences.
\begin{lemma}\label{l:seq-conv}
  Suppose $f$ is a step-isometry on $\cds$, that $(x_n)$ is a
  sequence in $\cds$ converging to $x$, and that $f(x_n)$ converges to
  $x'$. Then, for any $y\in \cds$ and $k\in \N$ which satisfy $\|x-y\|<k$
  we have $\|x'-f(y)\|\le k$.
\end{lemma}
\begin{proof}
  This is trivial. Indeed, suppose $\|x-y\|<k$. Then, for all
  sufficiently large $n$, $\|x_n-y\|<k$. Thus, since $f$ is a step
  isometry, $\|f(x_n)-f(y)\|<k$. Hence $\|x'-f(y)\|\le k$.
\end{proof}

\begin{lemma}\label{l:level-3-step-isom}
  Suppose $f$ is a step-isometry on $\cds$, that $(x_n),(y_n)$ are
  sequences in $\cds$ converging to $x$ and $y$ respectively, and that
  $f(x_n),f(y_n)$ converge to $x'$ and $y'$ respectively. Then, for
  any $k\in \N$, $\|x-y\|<3k$ if and only if $\|x'-y'\|< 3k$.

  In particular, for any $k\in \N$, if $y\in \cds$ then $\|x-y\|<3k$
  if and only if $\|x'-f(y)\|< 3k$.
\end{lemma}
\begin{proof} 
  Suppose that $\|x-y\|<3k$. Since $\cds$ is dense, we can pick $s,t\in
  \cds$ such that $\|x-s\|<k$, $\|s-t\|<k$ and $\|t-y\|<k$. By
  Lemma~\ref{l:seq-conv} $\|x'- f(s)\|\le k$ and $\|f(t)-y'\|\le
  k$. Also, since $f$ is a step-isometry on $\cds$, $\|f(s)-
  f(t)\|<k$. Hence, by the triangle inequality, $\|x'-y'\|<3k$.

  We obtain the converse by applying the above to $f^{-1}$ which is
  also a  step-isometry on $\cds$.

  The final part follows by taking the sequence $(y_n)$ to be the
  constant sequence $y$.
\end{proof}

\begin{lemma}\label{l:injective}
  Suppose $f$ is a step-isometry on $\cds$, that $(x_n),(y_n)$ are
  two sequences in $\cds$ converging to $x$, and that $f(x_n),f(y_n)$
  converge to $x'$ and $y'$ respectively. Then $x'=y'$.
\end{lemma}
\begin{proof}
  Suppose that $x'\not =y'$. Then the set
  \[
  \{v\in V: \|x'-v\|<3 \text{ and }\|y'-v\|>3\}
  \]
  is open and non-empty. Since $\cds$ is dense in $V$, there exists
  $z'\in \cds$ with $\|x'-z'\|<3$ and $\|y'-z'\|>3$. Let
  $z=f^{-1}(z')$. Then, Lemma~\ref{l:level-3-step-isom} applied to the
  sequences $(x_n)$ and $(y_n)$ implies $\|x-z\|<3$ and $\|x- z\|\ge 3$
  which is a contradiction.
\end{proof}

\begin{lemma}\label{l:convergence}
  Suppose that $(x_n)$ is a sequence in $\cds$ that converges in $V$. Then
  $f(x_n)$ is a convergent sequence.
\end{lemma}
\begin{proof}
  Since $(x_n)$ is convergent, there is an $m$ such that, for all $n>m$,
  we have $\|x_n-x_m\|<1$. Hence, since $f$ is a step-isometry,
  $\|f(x_n)-f(x_m)\|<1$ for all $n>m$; i.e., $f(x_n)$ is a bounded sequence. Thus,
  since $V$ is finite-dimensional, there is a subsequence $(x_{n_i})$
  such that $f(x_{n_i})$ converges to some value $x'$ say.

  Suppose that $f(x_n)$ does not converge to $x'$. Then there exists a
  subsequence bounded away from $x'$. As above we can take a further
  subsequence which converges and is bounded away from $x'$; in
  particular it must converge to some value $x''\not=x'$. But this
  contradicts Lemma~\ref{l:injective}.
\end{proof}
\begin{corollary}\label{c:bijective-S}
  Suppose $f$ is a step-isometry on $\cds$. Then there is a
  unique continuous function $\bar f\colon V\to V$ that extends $f$. 
\end{corollary}
\begin{proof}
  For any $x\in V$ define $\bar f(x)$ as follows. Choose a sequence
  $(x_n)$ in $\cds$ converging to $x$, and let $\bar
  f(x)=\lim_{n\to\infty}f(x_n)$. This limit exists by
  Lemma~\ref{l:convergence} and the function is well defined by
  Lemma~\ref{l:injective}.
  
  Finally, it is easy to see that $\bar f$ is continuous. Indeed,
  suppose that $(x_n)$ is a sequence in $V$ converging to $x$, say.  By
  the definition of $\bar f$ we can pick a sequence $(x_n')$ in $\cds$
  such that $\|x_n-x_n'\|<1/n$ and $\|f(x_n')-\bar f(x_n)\|<1/n$ for
  all $n$. Then $x_n'\to x$ so, since $\bar f$ is well defined,
  $f(x_n')\to f(x)$ and, thus, $\bar f(x_n)\to f(x)$ as required.
\end{proof}
\begin{corollary}\label{c:bijective-V}
  Any step-isometry on $V$ is continuous.
\end{corollary}
\begin{proof}
  This follows immediately from Corollary~\ref{c:bijective-S} by
  taking the dense set $S$ to be the whole of $V$.
\end{proof}

Lemma~\ref{l:level-3-step-isom} shows that $\bar f$ is a `scaled'
step-isometry; i.e., a step isometry in the norm $\frac13
\|\cdot\|$. Whilst that would be sufficient for our needs, $\bar f$ is
actually a step isometry in the original norm and we prove that
next. We start with the following trivial fact
\begin{lemma}\label{l:converge-either-side}
  Suppose that $S$ is a dense set in $V$ and that $x,y\in V$. Then
  there exist sequences $(x_n),(y_n)$ of points in
  $S$ converging to $x$ and $y$ respectively such that
  $\|x_n-y_n\|>\|x-y\|$ for all $n$. Similarly, providing $x\not= y$,
  we may choose such sequences $(x_n),(y_n)$ such that
  $\|x_n-y_n\|<\|x-y\|$ for all $n$.
\end{lemma}
\begin{proof}
  Let $\eps>0$. Let $r=\|x-y\|$. The point $y'=x+(1+\eps)(y-x)$ has
  $\|y-y'\|=\eps r$ and $\|x-y'\|=(1+\eps)r$. Let $x''$ be any point
  of $S$ in the set $B(x,\eps r/2)$ and $y''$ any point of $S$ in
  $B(y',\eps r/2)$.  By the triangle inequality, we have
  $\|x-x''\|<\eps r/2$, $\|y-y''\|<3\eps r/2$ and, also, $\|x''-y''\|>r$.

  We get the required sequence by setting $x_n,y_n$ to be the points
  $x'',y''$ given by the above argument when $\eps=1/n$.

  The second inequality is very similar but this time we choose
  $y'=x+(1-\eps)(y-x)$.
\end{proof}
\begin{proposition}\label{p:extends}
  The function $\bar f$ defined above is a step-isometry. Moreover,
  $\bar f$ preserves integer distances.
\end{proposition}
\begin{proof}
  Suppose $x$ and $y$ have $\|x-y\|\ge k$ for some $k\in \Z$. Then, by
  Lemma~\ref{l:converge-either-side} we can find sequences $(x_n)$ and
  $(y_n)$ in $S$ that converge to $x$ and $y$ respectively and have
  $\|x_n-y_n\|> k$. Hence, since $f$ is a step-isometry, $\|\bar
  f(x)-\bar f(y)\|=\lim_n\|f(x_n)-f(y_n)\|\ge k$.

  Similarly, if $x$ and $y$ have $\|x-y\|\le k$, then, by taking
  sequences with $\|x_n-y_n\|< k$, we see that $\|\bar f(x)-\bar f(y)\|\le
  k$. 

  This shows that if $\|x-y\|\in (k,k+1)$ then $\|\bar f(x)-\bar
  f(y)\|\in [k,k+1]$. Also, if $\|x-y\|=k$ then $\|\bar f(x)-\bar
  f(y)\|=k$; i.e., $\bar f$ preserves integer distances.

  Observe that $f^{-1}$ is also a step-isometry on $S$, so it extends to
  $\overline {f^{-1}}$ a step-isometry on $V$. Since we have
  $\overline{f^{-1}}\circ \bar f=\bar f
  \circ\overline{f^{-1}}=\textrm{id}$ on $S$, and $\bar f$ and
  $\overline{f^{-1}}$ are both continuous, we see that $\bar
  f^{-1}=\overline {f^{-1}}$. Thus, if $\|\bar f(x)-\bar f(y)\|=k$
  then $\|x-y\|=k$ and the result follows.
\end{proof}
\begin{corollary}\label{c:closed-balls}
  Suppose $f$ is a step-isometry on $V$. Then $f$ preserves integer
  distances. Moreover, for any integer $k$ and $x\in V$ we have
  $f(B(x,k))=B(f(x),k)$.
\end{corollary}
\begin{proof}
  For the first part, take $S=V$ in Proposition~\ref{p:extends}. By
  the definition of a step isometry, $f$ maps the open ball
  $B^\circ(x,k)$ to $B^\circ(f(x),k)$ so, since it and its inverse
  preserve integer distances, the second part follows.
\end{proof}

\section{Extreme points}\label{s:extreme}
For this section we assume $f$ is a (necessarily continuous)
step-isometry on all of $V$ that fixes $0$. The assumption that $0$ is
fixed makes the results simpler to state and this case is sufficient
for our needs.

Our aim in this section is to prove that $f$ maps the extreme points
of the unit ball to themselves, and that restricted to these extreme
points it is an isometry.

First we characterise the extreme points of $B$ in a purely norm/metric
way.
\begin{lemma}\label{l:nx}
  Suppose that $x$ is an extreme point of $B=B(0,1)$ and $n\in
  \N$. Then $B(0,1)\cap B(nx,n-1)=\{x\}$.
\end{lemma}
\begin{proof}
  Suppose that  $y\in B$ and $\|nx-y\|\le n-1$. Then
  $\|y\|\le 1$ and $\|\frac{n}{n-1}x-\frac{1}{n-1}y\|\le 1$. Since
  \[
  x=\frac{n-1}{n}\left(\frac{n}{n-1}x-\frac{1}{n-1}y\right)
    +\frac{1}{n}y
    \] and $x$ is an extreme point of $B$ we see that $y=x$.
\end{proof}
\begin{lemma}\label{l:characterize-extreme}
  A point $x$ in the unit ball $B=B(0,1)$ is an extreme point if and
  only if there exists a point $z$ such that $B(z,1)\cap
  B(0,1)=\{x\}$.
\end{lemma}
\begin{proof}
  If $x$ is an extreme point then the point $z=2x$ is such a point by
  Lemma~\ref{l:nx}.

  Now suppose that $z$ is a point such that $B(z,1)\cap B(0,1)=\{x\}$.
  Let $y=z-x$. Then $\|y\|\le 1$. Hence, the point $y$ is in $B(0,1)$
  and $B(z,1)$. Thus, since $x$ is the unique point in the
  intersection, $y=x$, so $z=2x$.

  Now suppose that $x=\frac12(y+w)$ for some $y,w\in B$. Then
  $2x-y=w\in B$, so $y\in B(0,1)\cap B(2x,1)$. Using the fact that $x$
  is the unique point in this intersection again, we have $y=w=x$, and
  we see that $x$ is an extreme point of $B$.
\end{proof}
We use this characterisation of the extreme points to show that $f$
maps them among themselves.
\begin{corollary}\label{c:extreme-to-extreme}
  The extreme points of the unit ball map to themselves under~$f$. 
\end{corollary}
\begin{proof}
  Lemma~\ref{l:characterize-extreme} characterises the extreme points
  by their integer distance properties. These are preserved by the
  step-isometry so the extreme points must be. Indeed, suppose $x$ is
  an extreme point of $B$. Then by Lemma~\ref{l:nx} the point $2x$ has
  the property that $B(0,1)\cap B(2x,1)=\{x\}$. Hence, by
  Corollary~\ref{c:closed-balls}, $B(0,1)\cap B(f(2x),1)$ must be the single point
  $f(x)$. Thus, by Lemma~\ref{l:characterize-extreme}, $f(x)$ is an
  extreme point of $B$.
\end{proof}
The final aim in this section is to show that $f$ restricted to the
extreme points of $B$ is an isometry.
\begin{lemma}\label{l:multiple-extreme}
  Suppose that $n\in \N$ and that $x$ is an extreme point of $B$. Then
  $f(nx)=nf(x)$.
\end{lemma}
\begin{proof}Obviously $f$ is also a step-isometry in the norm
  $\frac1n\|\cdot\|$ which has unit ball $nB$. Thus, since $nx$ is an
  extreme point of $nB$, it must map to a point $ny$ which is an
  extreme point of $nB$ and, thus, $y$ is an extreme point of $B$. We
  need to show that $f(x)=y$.

  By Lemma~\ref{l:nx}, $B(0,1)\cap B(nx,n-1)=\{x\}$. Hence, by
  Corollary~\ref{c:closed-balls}, $B(f(0),1)\cap
  B(f(nx),n-1)=\{f(x)\}$. Since $f(0)=0$ and $f(nx)=ny$,
  Lemma~\ref{l:nx} again shows that
  \[B(f(0),1)\cap
  B(f(nx),n-1)=B(0,1)\cap B(ny,n-1)=\{y\},
  \]
  and, thus, $f(x)=y$ as required.
\end{proof}
The next lemma provides a useful criterion for certain distances to be
preserved.
\begin{lemma}\label{l:multiple-isom}
  Suppose $x,y\in V$ have the property that $f(nx)=nf(x)$ and
  $f(ny)=nf(y)$ for any $n\in \N$. Then $\|x-y\|=\|f(x)-f(y)\|$.
\end{lemma}
\begin{proof}
  By hypothesis, for any $n\in \N$,
  \[\|f(x)-f(y)\|=\frac1n \|f(nx)-f(ny)\|.\]
  Also, since $f$ is a step-isometry
  \[ 
  \left\lfloor \|nx-ny\|\right\rfloor=\left\lfloor
    \|f(nx)-f(ny)\|\right\rfloor,
  \]
in particular 
\[
  \bigl|\|f(nx)-f(ny)\|-\|nx-ny\|\bigr|< 1.
\]
Hence 
\[\|f(x)-f(y)\|=\lim_{n\to \infty}\frac1n \|f(nx)-f(ny)\|=\lim_{n\to\infty}\frac1n \|nx-ny\|=\|x-y\|.\qedhere
\]
\end{proof}

\begin{proposition}
  The function $f$ is an isometry on the extreme points of~$B$.
\end{proposition}
\begin{proof}
  Suppose $x$ and $y$ are extreme points of $B$. We know that they map
  to extreme points.  By Lemma~\ref{l:multiple-extreme} we know that
  $f(nx)=nf(x)$ and $f(ny)=nf(y)$ for all $n\in \N$. Hence, by
  Lemma~\ref{l:multiple-isom}, $\|x-y\|=\|f(x)-f(y)\|$. Since this is
  true for all $x,y\in \ext(B)$, $f$ is an isometry on $\ext(B)$.
\end{proof}

\section{The lattice generated by the extreme points}\label{s:lattice}
Throughout this section we assume that $f$ is a (continuous)
step-isometry of $V$ that fixes 0.  In the previous section we showed
that $f$ maps the extreme points of $B$ to themselves. Obviously the
same argument shows that $f$ maps the extreme points of $B(y,1)$ to
extreme points of $B(f(y),1)$.  We start this section by showing that
this mapping is the `same' mapping.

\begin{lemma}\label{l:same-y}
  Suppose $x$ is an extreme point of $B$. Then for any $y\in V$ we have
  $f(y+x)=f(y)+f(x)$.
\end{lemma}
\begin{proof}
  The point $y+x$ is an extreme point of $B(y,1)$ so, by
  Corollary~\ref{c:extreme-to-extreme}, $f(y+x)=f(y)+z$ for some
  extreme point $z\in B$ and, by Lemma~\ref{l:multiple-extreme},
  $f(y+nx)=f(y)+nz$ for all $n\in \N$. Now the pairs of points $nx$
  and $y+nx$ are each $\|y\|$ apart: in particular these distances are
  bounded. Thus, since $f$ is step-isometry, the same is true of the
  pairs $f(nx)=nf(x)$ and $f(y+nx)=f(y)+nz$. Hence $z=f(x)$ as
  claimed.
\end{proof}
\begin{corollary}\label{c:neg-of-ext}
  For any extreme point $x$ of $B$ we have $f(-x)=-f(x)$.
\end{corollary}
\begin{proof}
  This is instant from Lemma~\ref{l:same-y}. Indeed 
\[0=f(0)=f(x+(-x))=f(x)+f(-x).\qedhere
\]
\end{proof}
Next we show that $f$ behaves well on the lattice $\Lambda$. (Recall
from Section~\ref{s:notation} that $\Lambda$ denotes the `lattice'
generated by the extreme points of $B$.)
\begin{corollary}\label{c:Lambda-preserved}
  The function $f$ maps $\Lambda$ to itself with
  \[
  f\left(\sum_{i=1}^n \lambda_ix_i\right)=\sum_{i=1}^n \lambda_i f(x_i)
  \]
  for any $\lambda_i\in\Z$ and $x_i\in \ext(B)$.  Moreover, for any $x\in
  \Lambda$ and $y\in V$, we have $f(y+x)=f(y)+f(x)$.
\end{corollary}
\begin{proof}
  Both parts follow by applying Lemma~\ref{l:same-y} and
  Corollary~\ref{c:neg-of-ext} repeatedly.
\end{proof}
\begin{lemma}\label{l:lattice-isometry}
  $f$ restricted to $\Lambda$ is an isometry.  
\end{lemma}
\begin{proof}
  By Corollary~\ref{c:Lambda-preserved}
  \[
  f\left(\sum_{i=1}^n \lambda_ix_i\right)=\sum_{i=1}^n \lambda_i f(x_i).
  \]
  In particular for any $n\in \N$ and $x\in \Lambda$ we have
  $f(nx)=nf(x)$. Thus Lemma~\ref{l:multiple-isom} shows that, for any
  $x,y\in \Lambda$, we have $\|x-y\|=\|f(x)-f(y)\|$; i.e., $f$ is an
  isometry on $\Lambda$.
\end{proof}
Of course this isometry extends from $\Lambda$ to $\overline\Lambda$.
\begin{corollary}
  $f$ restricted to the closure  $\overline\Lambda$ of $\Lambda$ is an additive isometry.
\end{corollary}
\begin{proof}
  $f$ is continuous and is an additive isometry on $\Lambda$.
\end{proof}
Our final aim in this section is to show that there exists an isometry
$Q$ of $V$ such that $Q\circ f$ fixes $\Lambda$ pointwise. Obviously
$Q\circ f$ is also a step-isometry so in our later arguments we are
able to reduce to the case when $f$ fixes $\Lambda$.
\begin{lemma}
  There exists a unique linear isometry $\hat f\colon V\to V$ such
  that $\hat f$ and $f$ agree on $\overline\Lambda$.
\end{lemma}
\begin{proof}
  First, define $\hat f$ on $\Q \Lambda$ by $f(qv)=qf(v)$ where $q\in
  \Q$ and $v\in \Lambda$. This is well defined and linear since $f$ is
  additive on $\Lambda$. Since $f$ is an isometry on $\Lambda$, $\hat
  f$ is an isometry on $\Q \Lambda$. Now, since $\Lambda$ is spanning,
  $\Q\Lambda$ is dense in $V$ and thus $\hat f$ extends to a linear
  isometry on $V$.

  The uniqueness is trivial since $\Lambda$ is spanning.
\end{proof}

\begin{corollary}\label{c:isometry}
  There exists an isometry $Q$ of $V$ such that $Q \circ f$ fixes
  $\Lambda$ pointwise. 
\end{corollary}
\begin{proof}
  Let $Q$ be the isometry extending $f^{-1}$, as guaranteed by the
  previous lemma. Then $Q\circ f$ fixes $\Lambda$ pointwise.
\end{proof}

\section{Extreme lines and preserved directions}\label{s:extreme-lines}
In this section we assume that $f$ is a step-isometry of $V$ that
fixes~$\Lambda$ pointwise, and so, in particular, $f(0)=0$.

Our aim in this section is to show that many directions are unchanged,
or `preserved'.
\begin{defn}
  A \emph{preserved direction} is a vector $x$ such that, for all
  $\alpha\in \R$ and for all $y\in V$, the vector $f(y+\alpha x)-f(y)$
  is a multiple of $x$.
\end{defn}
\noindent%
In particular, since we are assuming $f(0)=0$, for any preserved
direction $x$, $f(x)$ is a multiple of $x$.

Preserved directions turn out to be closely related to
extreme lines, which are a standard generalisation of the notion of
extreme points.
\begin{defn}
  Suppose $A$ is a convex body. An \emph{extreme line} of $A$ is a
  line segment $[x,y]$ in $A$ such that, for all $z\in [x,y]$, if $z$ is a
  convex combination of $s,t \in A$ then $s,t\in [x,y]$.
\end{defn}
\begin{remark}
  Obviously, if $[x,y]$ is an extreme line then $x$ and $y$ are
  extreme points of $A$.
\end{remark}

Just as extreme points are characterised by the intersection
properties of balls, so are extreme lines.
\begin{lemma}\label{l:ext-lines=ball}
  Suppose $[x,y]$ is an extreme line of the unit ball $B=B(0,1)$. Then
  \[[x,y]=B(0,1)\cap B(x+y,1).\]
\end{lemma}
\begin{proof}
  Since $x,y\in B$, we have $x,y\in B(x+y,1)$; i.e., $x,y\in B\cap
  B(x+y,1)$. Hence, by convexity, $[x,y]\subseteq B\cap B(x+y,1)$.
  
  Suppose that $z\in B\cap B(x+y,1)$. Then $z\in B$ and $x+y-z\in B$. 
  Thus
  \[\frac{x+y-z}2+\frac{z}{2}=\frac{x+y}2
  \]
  is a point in $[x,y]$ that is a convex combination of points in
  $B$. Since $[x,y]$ is an extreme line this implies that $z \in
  [x,y]$.
\end{proof}

We will be interested in the directions of the extreme lines rather
than the lines themselves. Thus we make the following definition.
\begin{defn}
  Suppose $B$ is the unit ball of a normed space $V$. An \emph{extreme
    line direction} is any non-zero multiple of the vector $x-y$ where
  $[x,y]$ is an extreme line in $B$.
\end{defn}
\begin{remark}
  We view extreme line directions that are (non-zero) multiples of
  each other as the \emph{same} extreme line direction.
\end{remark}

The key result for preserved directions is that all extreme line
directions are preserved directions.
\begin{proposition}\label{p:ext-line=preserved-line}
  Suppose $B$ is the unit ball and $[x,y]$ is an extreme line. Then
  $x-y$ is a preserved direction.
\end{proposition}
\begin{proof}
  Suppose $v_1,v_2\in V$ satisfy $v_2=v_1+\alpha(x-y)$ for some
  $\alpha>0$. Let $n=\lceil \alpha \rceil$ and $u=v_1-n  x$. Then we have
$v_1,v_2\in u+[nx,ny]$.

  Now, by Lemma~\ref{l:ext-lines=ball}, for any point $z\in u+[nx,ny]$
  we have $z\in B(u,n)\cap B(u+nx+ny,n)$.  Hence, since $f$ is a
  step-isometry, $f(z)\in B(f(u),n)\cap B(f(u+nx+ny),n)$. Since
  $nx+ny\in \Lambda$, by Corollary~\ref{c:Lambda-preserved}, we have
  $f(u+nx+ny)=f(u)+nx+ny$. Thus,
  \[f(z)\in B(f(u),n)\cap B(f(u)+nx+ny,n)=f(u)+[nx,ny]
  \]
  by Lemma~\ref{l:ext-lines=ball} again.  In particular both $f(v_1)$
  and $f(v_2)$ lie in $f(u)+[nx,ny]$. Thus
  \[f(v_2)-f(v_1)=\beta (x-y)\] for some $\beta$, as claimed.
\end{proof}
\begin{remark}
  The map $f$ need not preserve the directions of the extreme points:
  indeed consider the $\ell_\infty^2$ case where $f$ can treat each
  coordinate separately and, thus, need not preserve the line $y=x$
  through the extreme point $(1,1)$.
\end{remark}
\section{Strongly fixed subspaces}\label{s:fixed-subspace}
In this section we assume that $f$ is a step-isometry of $V$ that
fixes~$\Lambda$ pointwise.

\begin{defn}
  We say a subspace $U$ of $V$ is \emph{strongly fixed} if, for all
  $u\in U$ and $v\in V$, we have $f(u+v)=u+f(v)$.
\end{defn}
\begin{remark}
  It is immediate from the definition that if $U$ and $U'$ are strongly fixed
  subspaces then $U+U'$ is a strongly fixed subspace.
\end{remark}

We have seen (Corollary~\ref{c:Lambda-preserved}) that $f(u+v)=u+f(v)$
for all $u\in \overline\Lambda$ and $v\in V$. Hence, the continuous
subspace $U_0$ of $\overline\Lambda$ is a strongly fixed subspace. Our
aim in the next two sections is to show that the whole of $U$ in the
$\ell_\infty$-decomposition of $V$ is strongly fixed; in this section
we show that a `large' subspace is strongly fixed. Then, in the next
section, we show that what is left is essentially an $\ell_\infty$
subspace -- in particular, that it is spanned by
$\ell_\infty$-directions.

\begin{lemma}\label{l:indep-preserved}
  Suppose $x_1,x_2,\dots x_k$ is a linearly independent set of
  preserved directions. Then
  \[
  f\left(\sum_{i=1}^k \lambda_ix_i\right)=\sum_{i=1}^k f(\lambda_ix_i)
  \]
  for any $\lambda_1,\lambda_2,\dots,\lambda_k\in \R$.
\end{lemma}
\begin{proof}
  We prove this by induction on $k$. It is trivial for $k=1$.

  Suppose it is true for $k-1$: i.e.,
  \[
  f\left(\sum_{i=1}^{k-1} \lambda_ix_i\right)=\sum_{i=1}^{k-1} f(\lambda_ix_i).
  \]
  Since $\sum_{i=1}^k \lambda_ix_i-\sum_{i=1}^{k-1}
  \lambda_ix_i=\lambda_kx_k$ which is a preserved direction we see
  that 
  \[
  f\left(\sum_{i=1}^{k} \lambda_ix_i\right)
  =f\left(\sum_{i=1}^{k-1} \lambda_ix_i\right)+\mu_k x_k
  =\sum_{i=1}^{k-1} f(\lambda_ix_i)+\mu_k x_k
  \]
  for some $\mu_k$.

  Similarly by applying the induction hypothesis to the last $k-1$
  summands rather than the first we see that
  \[
  f\left(\sum_{i=1}^{k} \lambda_ix_i\right)
  =f\left(\sum_{i=2}^{k} \lambda_ix_i\right)+\mu_1 x_1
  =\sum_{i=2}^{k} f(\lambda_ix_i)+\mu_1 x_1.
  \]
  The $x_i$ are preserved directions so $f(\lambda_ix_i)$ is a
  multiple of $x_i$ for each $i$. Thus, since the $x_i$ are linearly
  independent, we see that $\mu_k x_k=f(\lambda_k x_k)$ as required.
\end{proof}

\begin{lemma}\label{l:min-dep}
  Suppose $x_1,x_2,\dots x_k$ form a minimal linearly dependent set of
  preserved directions, and that $k\ge 3$. Then
  $\left<x_1,x_2,\dots,x_k\right>$ is a strongly fixed subspace.
\end{lemma}
\begin{proof}
  Suppose that $\sum_{i=1}^k \lambda_i x_i=0$ is a non-trivial linear
  dependence. Since the $x_i$ form a minimal linear dependent set all
  the $\lambda_i$ are non-zero. Thus we may assume $\lambda_1=1$.

  We start by showing that for any $m\in \N$ we have $f(m x_1)=m
  f(x_1)$. We prove this by induction. The case $m=1$ is trivial so
  suppose that $f((m-1)x_1)=(m-1)f(x_1)$.
  We have 
\begin{align*}
  f(mx_1)&=f\left((m-1)x_1-\sum_{i=2}^k\lambda_ix_i\right)\\
  &=f\left((m-1)x_1-\sum_{i=2}^{k-1}\lambda_ix_i\right) + Cx_k&\text{(for some $C$, since $x_k$ is a preserved direction)}\\
  &=f((m-1)x_1)+f\left(\sum_{i=2}^{k-1}-\lambda_ix_i\right) + Cx_k\intertext{(by Lemma~\ref{l:indep-preserved} twice, since $x_1,\dots,x_{k-1}$ are linearly independent)}
  &=f((m-1)x_1)+f\left(\sum_{i=2}^k-\lambda_ix_i\right) + C'x_k&\text{(for some $C'$, since $x_k$ is a preserved direction)}\\
  &=(m-1)f(x_1)+f(x_1)+C'x_k&\text{by the inductive hypothesis}\\
  &=mf(x_1)+C'x_k.
\end{align*}
But since $x_1$ is a preserved direction and $x_1,x_k$ are linearly
independent $C'=0$ and the induction is complete. 

Obviously, $\alpha x_1$ is also a preserved direction for any
$\alpha\not=0$, so the above shows that $f(\alpha x_1)=\alpha f(x_1)$
for all $\alpha\in\Q$ with $\alpha>0$. Since $f$ is continuous this
means that $f(\alpha x_1)=\alpha f(x_1)$ for all $\alpha>0$. 

Now, for any $\alpha>0$, by Lemma~\ref{l:lambda-cover} there is a point
$y\in \Lambda$ with $\|\alpha x_1-y\|\le \dim V/2$. Thus, since $f$ is a
step-isometry, we have 
\[
\|f(\alpha x_1)-f(y)\|\le \|\alpha x_1-y\|+1\le \dim V/2+1.
\] 
But $f$ fixes $\Lambda$ pointwise so $f(y)=y$ and, thus, $\|f(\alpha
x_1)-\alpha x_1\|$ is bounded independently of $\alpha$.  Since,
$\|f(\alpha x_1)-\alpha x_1\| =\alpha \|f(x_1)-x_1\|$ this implies
that $f(x_1)=x_1$ and, thus, that $f(\alpha x_1)=\alpha x_1$ for all
$\alpha>0$. The same argument applied to $-x_1$ -- obviously also a
preserved direction -- shows that $f(-\alpha x_1)=-\alpha x_1$. This
shows that $f$ is the identity on $\spanof{x_1}$.

We have shown that $f$ fixed $\spanof{x_1}$ pointwise, but we want to
show more: that $f$ strongly fixes $\spanof{x_1}$. For any $v\in V$,
the function $g$ defined by $g(x)=f(x+v)-f(v)$ is also a step-isometry
and, by Corollary~\ref{c:Lambda-preserved}, fixes $\Lambda$. Moreover,
$g$ also preserves the directions $x_i$. Thus, by the above argument
$g$ is the identity on $\spanof{x_1}$. Hence $f(v+\alpha
x_1)=f(v)+\alpha x_1$ for all $\alpha\in \R$ i.e., $\spanof{x_1}$ is a
strongly fixed subspace.

Since this is true for each $x_i$, we see that $\spanof{x_1,x_2,\dots\
  x_k}$ is a strongly fixed subspace.
\end{proof}

The previous lemmas show that the span of linearly dependent preserved
directions is strongly fixed. Of course, we also know that the
continuous subspace $U_0$ of $\overline{\Lambda}$ is strongly
fixed. Thus we make the following definition to cover the largest
subspace that we know (so far) is strongly fixed. Later, we shall show
that this is the non-$\lid$-component of the
$\ell_\infty$-decomposition.

Before stating the main definition we need a little more
notation. Suppose that $W$ is any subspace of $V$ and $x_1,x_2,\dots,
x_k$ are vectors in $V$. A linear combination of the $x_i$ over $W$ is
any sum of the form $w+\sum_i\lambda_ix_i$, where $w\in W$; the
\emph{span of the $x_i$ over $W$} is $\spanof{W, x_1,x_2,\dots,x_k}$;
the $x_i$ are \emph{linearly independent over $W$} if $\sum_i\lambda
x_i\in W$ implies that $\lambda_i=0$ for all $i$.

\begin{defn}
  Suppose that $V$ is a normed space with unit ball $B$, that $U$ is a
  subspace and $x_i$, $i\in I$ are the extreme line directions in
  $U$. Then $U$ is \emph{well-spanned} if
  \begin{itemize}
    \item it contains the continuous subspace $U_0$ of $\Lambda$
    \item the $x_i$ span $U$ over $U_0$ 
    \item every $x_i\in U\setminus U_0$ can be written as a linear combination of the
    other~$x_j$ over~$U_0$
  \end{itemize}
\end{defn}
First, we show that there is a unique maximal well-spanned subspace and
then that any step-isometry that pointwise fixes $\Lambda$ strongly
fixes this subspace.
\begin{lemma}\label{l:eld-li}
  Suppose that $V$ is a normed space with unit ball $B$. Then there is
  a unique maximal well-spanned subspace $U$. Moreover, the extreme
  line directions outside $U$ are linearly independent over $U$.
\end{lemma}
\begin{proof}
  Obviously $U_0$ is well-spanned. Moreover, if $U$ and $U'$ are
  well-spanned then so is $U+U'$. Thus there is a unique maximal
  well-spanned subspace.

  To prove the second part let $U$ be the maximal well-spanned
  subspace and $x_1,x_2,\dots,x_k$ be the extreme line directions in
  $U$.  Suppose that $y_1,y_2,\dots,y_l\in V\setminus U$ is a minimal
  linearly dependent set of extreme line directions over $U$. Then,
  since the $x_i$ span $U$ over $U_0$, we see that, for each $i$,
  $y_i$ can be written as a linear combination of the $\{x_j:1\le j\le
  k\}\cup\{y_j:j\not =i\}$ over $U_0$.  Hence
  $U+\spanof{y_1,y_2,\dots,y_\ell}$ is a well-spanned subspace
  contradicting the maximality of $U$.
  \end{proof}

\begin{corollary}\label{c:well-spanned-fixed}
  Suppose that $V$ is a normed space with maximal well-spanned
  subspace $U$ and that $f$ is a step-isometry fixing
  $\Lambda$. Then $U$ is a strongly fixed subspace.
\end{corollary}
\begin{proof}
  We have seen that $f$ strongly fixes $U_0$. Consider any extreme
  line direction $v$ in $U$. Then $v$ occurs in a minimal linear
  relation with other extreme line directions in $U$ over
  $U_0$. Since, by Proposition~\ref{p:ext-line=preserved-line},
  extreme line directions are preserved directions of $f$,
  Lemma~\ref{l:min-dep} shows that $f$ is strongly fixed on the span
  of these directions and, in particular, on $\spanof{v}$. Since this
  is true for every extreme line direction in $U$, and these
  directions span $U$ over $U_0$, we see that $U$ is strongly fixed.
\end{proof}

\section{The complement of the maximal well-spanned subspace}\label{s:complement=well-spanned}
In this section we prove that $V=(U\oplus \spanof{v_1}\oplus
\spanof{v_2}\dots\spanof{v_k})_\infty$ where $U$ is the maximal
well-spanned subspace and $v_1\dots v_k$ are extreme line directions
outside of $U$ and, thus, deduce that $U$ is the non-$\lid$-component
in the $\ell_\infty$-decomposition.

We start by showing that, unless $U=V$, there is an extreme line
direction outside of~$U$. Since we use induction it is convenient to
prove a (stronger) result for a general convex set rather than just
for the unit ball of the normed space.
\begin{lemma}\label{l:exists-preserved}
  Suppose $U$ is a codimension one subspace of $V$, that $v\in
  V\setminus U$, and that $U_i=U+\lambda_iv$, $1\le i\le k$, are
  distinct cosets of $U$ with $\lambda_1<\lambda_2<\dots<\lambda_k$.
  Further, suppose that, for each~$i$, $A_i$ is a (non-empty) compact
  convex subset of $U_i$, and that, for some $s<k$, $x\in A_s$ is an
  extreme point of $A=\conv(\bigcup_i A_i)$.  Then there exists $t>s$
  and $y\in A_t$ such that $[x,y]$ is an extreme line of~$A$.
\end{lemma}
\begin{proof}
  We prove this by induction on the dimension of $V$. If $\dim V=1$ it
  is trivial: $V=\R$ and each $A_i$ is a single point. Since $s<k$ and
  $x$ is extreme point we must have $x\in A_1$ so join it to the point
  in $A_k$. Thus suppose that the result holds for all spaces of
  dimension less than $\dim V$.

  Let $H_0$ be a codimension one tangent hyperplane at $x$ to $A$ and
  let $h_0$ be a corresponding linear functional; i.e., $H_0=\{y\in
  V:h_0(y)=h_0(x)\}$. We may assume that $h_0(y)\le h_0(x)$ for all
  $y\in \bigcup_i A_i$.

  Let $q$ be the linear functional on $V$ defined, for any $u\in U$
  and $\lambda$, by $q(u+\lambda v)=\lambda$. By hypothesis
  $q(U_i)=\lambda_i$ is increasing with $i$.  Consider the family of
  hyperplanes $H_\alpha$ through $x$ given by the functionals
  $h_\alpha=h_0+\alpha q$; i.e., $H_\alpha=\{y\in
  V:h_\alpha(y)=h_\alpha(x)\}$. Let $H_\alpha^-=\{y\in
  V:h_\alpha(y)\le h_\alpha(x)\}$.  Note that, $A_i\subseteq H_0^-$ for
  all $i$.

  For each $i > s$, the function $\alpha_i(y) =
  (h_0(x) - h_0(y))/(\lambda_i - \lambda_s)$ is continuous and
  non-negative on the compact set $A_i$ and so attains an absolute
  minimum $\alpha^*_i \geq 0$.  Set $\beta = \min_{i> s}\{\alpha^*_i\}
  \geq 0$.  Then, by the choice of $\beta$, for every $i > s$, and $y
  \in A_i$, we have $h_\beta(y) \leq h_\beta(x)$. Additionally, for
  every $i \le s$, and $y \in A_i$, since $\beta\ge 0$ and
  $\lambda_i\le \lambda_s$, we also have $h_\beta(y) \leq
  h_\beta(x)$. Thus, $\bigcup_{i} A_i\subseteq H^{-}_\beta$, so
  $H_\beta$ is a tangent hyperplane to $A$ at $x$.
  
  Furthermore, since all the minimums $\alpha_i^*$ were attained in
  the choice of $\beta$, there is at least one $j>s$ and $y \in A_j$
  with $h_\beta(y) = h_\beta(x)$ and so $H_\beta \cap (\bigcup_{i > s}
  A_i) \neq \emptyset$.

  Let $H=H_\beta$ and $H^-=H_\beta^-$, and, for each $i$, let
  $A_i'=A_i\cap H$.  Note some of the $A_i$ may be empty and we ignore
  these sets. Let
  \[A'=\conv(\bigcup_i A_i')=\conv (\bigcup_i A_i\cap H)=\conv
  (\bigcup_i A_i)\cap H=A\cap H\]
  where the third equality follows since  $\bigcup_iA_i\subset H^-$.
  Now each $A_i'$ lies in $U_i\cap H$ which are cosets of $U\cap H$
  which is codimension one. Obviously the $A_i'$ are compact convex
  subsets. Also $x\in A_s'$ and, since $A'\subset A$, we see that $x$
  is an extreme point of $A'$. Finally, by our choice of $H$, at least one of
  the $A_{s'}'$ for $s'>s$ is non-empty. Hence the $A_i'$ satisfy the
  induction hypothesis. Thus, there exists $y\in A_t'$ with $t>s$ such
  that $[x,y]$ is an extreme line of $A'$.

  To complete the induction step, and thus the proof, we show that
  $[x,y]$ is extreme line of $A$. Indeed, suppose $z\in[x,y]$ is a
  convex combination of $s,t\in A$. Since $[x,y]\subset A'\subset H$
  and $A\subset H^-$ both $s,t$ must lie in $H$, and thus $s,t\in
  A'$. Since $[x,y]$ is an extreme line in $A'$ this shows that
  $s,t\in [x,y]$, and thus $[x,y]$ is an extreme line of $A$ as
  claimed.
\end{proof}

We use this result to deduce that there are `many'  extreme line directions.

\begin{corollary}\label{c:preserved-span}
  Suppose that $U$ is the maximal well-spanned subspace of $V$. Then
  the extreme line directions outside $U$ span $V$ over $U$.
\end{corollary}
\begin{proof}
  If $U=V$ then the statement is (rather vacuously) true so assume
  $U\not=V$.  Since $\ext(B)$ spans $V$ there is an extreme point
  $x\not\in U$. Let $y_i$, $i\in I$ be the endpoints of the extreme
  lines $[x,y_i]$ which have $x$ as the other endpoint. If $U$
  together with the vectors $x-y_i$ span $V$ then the result holds, so
  suppose they do not.

  Let $U'$ be a codimension one subspace containing $U$ and all the
  vectors $x-y_i$.  Fix $v\in V \setminus U'$ and let
  $U'_1,U_2',\dots,U'_k$ be the cosets of $U'$ covering the extreme
  points of $B$, where $U_i'=U+\lambda_iv$ are such that the
  $\lambda_i$ are increasing. By Corollary~\ref{c:coset-cover}, such a
  $k$ exists and, since $B$ is not contained in any codimension one
  affine hyperplane, $k\ge 2$.

  By replacing $v$ with $-v$ (and thus reversing the order of the
  $U_i'$) if necessary, we may assume $x\in U'_s$ for some $s<k$.  Now
  apply the previous lemma with $U'$, taking the set $A_i$ in $U_i'$
  to be $B\cap U'_i$ for each $i$. Note that, since all the extreme
  points of $B$ are contained in $\bigcup_i A_i$, we have
  $\conv(\bigcup_i A_i)=B$.
  
  This gives an extreme line $[x,y]$ of $B$ with $x-y$ not in $U'$
  contradicting the choice of $U'$.
\end{proof}

\begin{lemma}\label{l:ext-line=l_inf-line}
  Suppose that $v_1$ is an extreme line direction not in the maximal
  well-spanned subspace $U$. Then $v_1$ is an $\ell_\infty$-direction
  and $U$ is a subset of $U_1$ the corresponding subspace.
\end{lemma}
\begin{proof}
  Let $v_2,\dots v_k$ be the other extreme line directions outside of
  $U$. We may assume that they all, and $v_1$, have norm one.  By
  Lemma~\ref{l:eld-li}, the $v_i$ are linearly independent over $U$,
  and by Corollary~\ref{c:preserved-span} they span over $U$.

  Let $U'$ be the subspace spanned by $U$ and $v_2,v_3,\dots,v_k$.
  Since the $v_i$ are linearly independent and span over $U$, we see
  that $U'$ has codimension one.

  Suppose that $U'_1,U'_2,\dots,U'_t$ are finitely many cosets
  (Corollary~\ref{c:coset-cover}) of $U'$ that cover the extreme
  points of~$B$. Our first step is to show that, from every extreme
  point of $B$, we can either add or subtract a multiple of $v_1$ and
  stay in $B$.

  Write $U'_i=U'+\lambda_i v_1$, and we may assume that the $\lambda_i$
  are increasing. Define $A_1,A_2,\dots,A_t$ by $A_i=B\cap U_i'$. Note
  that $B=\conv(\bigcup_i A_i)$.

  For any extreme point $x$ of $B$ in some $A_i$ with $i<t$,
  Lemma~\ref{l:exists-preserved} shows that there exists $y$ in one of
  the $A_s$ with $s>i$ such that $x-y$ is an extreme line
  direction. Since $v_1$ is the only extreme line direction not in $U'$
  we must have that $x-y$ is in the same direction as $v_1$. Thus,
  $y=x+\lambda v_1$ for some $\lambda$, and since $s>i$ we see
  $\lambda>0$.

  By applying Lemma~\ref{l:exists-preserved} again, this time to the
  $A_i$ in reverse order, we see that any extreme point $x'$ of $B$ in
  any of the $A_i$ with $i>1$ there is also a $y'\in A_s$ for some
  $s<i$ with $x'-y'$ an extreme line direction. Again $x'-y'$ must be
  the same direction as $v_1$; i.e., $y'=x'+\lambda' v_1$. This time,
  since $s<i$ we see that $\lambda'<0$.

  Since the extreme points of $B$ span $V$ and $B$ is symmetric, we
  see that $\ext(B)$ is not a subset of $U'$ or any single coset of
  $U'$, and thus $t\ge 2$. Hence, for any extreme point of $B$, at
  least one of the two cases above applies; i.e., we have shown that
  from any extreme point of $B$ we can either add or subtract a
  multiple of $v_1$ and stay in $B$.

  It now follows that $t=2$; i.e., that the extreme points of $B$ are
  contained in two cosets of $U'$. Indeed, suppose $t\ge 3$. By
  applying the two cases above to any extreme point $x$ in $A_2$ we
  see that $x+\lambda v_1$ and $x+\lambda'v_1$ are both in $B$ for
  some $\lambda>0$ and $\lambda'<0$. But this contradicts $x$ being an
  extreme point of $B$.

  Since $B$ is symmetric we must have $U_1'=U'-\lambda v_1$ and
  $U_2'=U'+ \lambda v_1$ for some $\lambda>0$.  Let $B_1=A_1+ \lambda
  v_1$ and $B_2=A_2- \lambda v_1$ be the projections of $A_1,A_2$ onto
  $U$.

  We claim that $B_1=B_2$. For a contradiction suppose there is a
  point in $B_2\setminus B_1$. Then there must be an extreme point $z$
  of $B_2$ in $B_2\setminus B_1$. Obviously $z'=z+\lambda v_1\in A_2$
  is an extreme point of $B$. However, since $z\not \in B_1$ we see
  that we can not add or subtract any multiple of $v_1$ to $z'$ and
  stay in $B$ which is a contradiction.

  Now since $v_1\in B$ (recall we assumed $\|v_1\|=1$) we see $\lambda\ge
  1$. Also for any $z\in A_1$ the vector $z+2\lambda v_1\in A_2$, so
  $z$ and $z+2\lambda v_1$ are both in $B$; in particular $\lambda\le
  1$. Thus $\lambda=1$.

  Combining this we see that $B=\conv(B_1+v_1,B_1-v_1)$. We use this
  to show that $v_1$ is an $\ell_\infty$-direction.  Given any $v\in
  V$ write $v=\alpha v_1+\beta u_1$ for some $\alpha,\beta\in \R$ and
  $u_1\in U'$ with $\|u_1\|=1$. Observe that the description of $B$
  above shows that $B\cap U'=B_1$. Thus, since $\|u_1\|=1$ we see that
  $u_1\in B$ so $u_1\in B_1$. 
  
  Now
  \[
  \|v\|=\inf\{\lambda: v/\lambda\in B\}=\max(\alpha,\beta)=\max(\alpha,\|\beta u_1\|).
  \]
  Since $U\subseteq U'$ the result follows.
\end{proof}

\begin{lemma}\label{l:l-inf=ext-line}
  Suppose that $x$ is an $\ell_\infty$-direction with corresponding
  subspace~$W$. Then $x$ is an extreme line direction, and the maximal
  well-spanned subspace, $U$, is contained in $W$.
\end{lemma}
\begin{proof}
  Let $B_{W}=B\cap W$ be the unit ball in $W$. We claim that
  $B=\conv(B_{W}+x,B_{W}-x)$. Suppose $v\in V$. Then, since $x$ is an
  $\ell_\infty$-direction, we can write $v=w+\lambda x$, and we have
  $\|v\|=\|w+\lambda x\|=\max(\|w\|,|\lambda|)$. This implies that, if
  $\|v\|\le 1$, then $v$ is a convex combination of $w+x$ and $w-x$,
  for some $w\in B_W$, i.e. $B\subseteq \conv (B_{W}+x,B_{W}-x)$; and,
  conversely, it implies that $B_{W}+x\subset B$ and $B_{W}-x\subset
  B$, so $\conv (B_{W}+x,B_{W}-x)\subseteq B$. This completes the
  proof of the claim.

  It is immediate that $x$ is an extreme line direction: indeed, for
  any $w\in \ext (B_W)$, $[w-x,w+x]$ is an extreme line.

  We also see that $\ext(B)\subset (W+x)\cup (W-x)$ so, in particular,
  the continuous subspace $U_0$ is contained in $W$. Moreover, the
  only extreme line direction outside $W$ is $x$. Indeed, suppose
  $[y_1,y_2]$ is an extreme line. If $y_1,y_2$ are both contained in
  $W+x$ or both in $W-x$ then $y_2-y_1\in W$. Thus assume $y_1\in
  W-x$ and $y_2\in W+x$. Write $y_1=z_1-x$ and $y_2=z_2+x$, so
  $z_1,z_2\in B_{W}$.  The point $\frac12(z_1+z_2)\in [y_1,y_2]$ is a
  convex combination of $z_1+x,z_2-x$. Since $[y_1,y_2]$ is an extreme
  line this shows that $z_1+x,z_2-x\in [y_1,y_2]$, and thus $z_1=z_2$
  and $y_2-y_1=2x$; i.e., the extreme line $[y_1,y_2]$ has direction
  $x$.

  In particular, this shows that $x$ is not a linear combination of
  other extreme line directions over $U_0$: i.e., $x$ is not in any
  well-spanned subspace. Moreover, since all other extreme line
  directions, and the continuous subspace, lie in $W$ we see that the
  maximal well-spanned subspace is contained in $W$.
\end{proof}
Finally, we show that the well spanned subspace is actually the
non-$\lid$-component in the $\ell_\infty$-decomposition.
\begin{proposition}\label{p:well-spanned=non-l-inf}
  Suppose that $V$ is a normed space with $\ell_\infty$-decomposition
  $(W\oplus \ell_\infty^d)_\infty$, and that $U$ is the maximal
  well-spanned subspace. Then $U=W$.
\end{proposition}
\begin{proof}
  Let $u_1,u_2,\dots,u_k$ be the extreme line directions outside $U$,
  and let $U'=\spanof{u_1,u_2,\dots,u_k}$. Let $w_1,w_2,\dots,w_d$ be
  all the $\ell_\infty$-directions, with corresponding subspaces
  $W_i$, and let $W'=\spanof{w_1,w_2,\dots, w_d}$ (i.e., $W'$ is
  the $\lid$-component in the $\ell_\infty$-decomposition).

  By Lemma~\ref{l:ext-line=l_inf-line}, each $u_i$ is an
  $\ell_\infty$-direction, so $U'\subseteq W'$.  Also, by
  Lemma~\ref{l:l-inf=ext-line}, $U\subset W_i$ for each $i$, so
  $U\subset \bigcap_{i=1}^d W_i=W$
  (Proposition~\ref{p:l-inf-decomposition}).  Since the sum $V=W\oplus
  W'$ is direct we must have $U=W$ (and $U'=W'$).
\end{proof}

\section{Proof of Theorem~\ref{t:step-isometry}}\label{s:step-isom}
Finally we are in a position to prove
Theorem~\ref{t:step-isometry}. We prove it first for the case when $f$
fixes $\Lambda$ pointwise.
\begin{lemma}\label{l:step-isometry-fixed-lambda}
  Suppose that $V$ is a normed space with $\ell_\infty$-decomposition
  $V=(U\oplus \ell_\infty^d)_\infty$ and that $f$ is a step-isometry
  fixing $\Lambda$ pointwise. Then $f$ factorises over the
  decomposition as $f_U\oplus f_{\ell_\infty^d}$ where $f_U$ is the
  identity on $U$ and $f_\lid$ is a step-isometry on
    $\ell_\infty^d$.
\end{lemma}
\begin{proof}
  Let $f_U$ be the identity on $U$ and define
  $f_{\ell_\infty^d}=f|_\lid$. We show that this is a factorisation of
  $f$ over the decomposition $U\oplus\lid$. By
  Proposition~\ref{p:well-spanned=non-l-inf}, $U$ is the maximal
  well-spanned subspace so, by Corollary~\ref{c:well-spanned-fixed},
  it is strongly fixed by $f$. Thus $f=f_U\oplus f_\lid$. Obviously
  $f_U$ maps $U$ to itself, so it remains to show that $f_\lid$ maps
  $\lid$ to itself.

  Let $v_1,v_2,\dots, v_d$ be the $\ell_\infty$-directions (i.e., the
  natural basis of the $\lid$-component). By
  Lemma~\ref{l:l-inf=ext-line}, each $v_i$ is an extreme line
  direction so, by Proposition~\ref{p:ext-line=preserved-line} a
  preserved direction. Suppose that
  $v=\sum_{i=1}^d\lambda_iv_i$. Then, inductively using the fact that
  each $v_i$ is a preserved direction, we have
  \[
  f_\lid(v)=f(v)=f(\sum_{i=1}^d\lambda_iv_i)=\sum_{i=1}^d\lambda_i'v_i
  \]
  for some $\lambda_i'$; i.e., $f_\lid$ does map $\lid$ to itself.

  It is easy to see that the factors in any factorisation of a
  bijection are also bijections. Thus, since $f_{\ell_\infty^d}$ is just
  the restriction of $f$ to $\lid$, we see that $f_\lid$ is a
  step-isometry as claimed.
  \end{proof}

\begin{proof}[Proof of Theorem~\ref{t:step-isometry}]
  We have that $f$ is any step-isometry on $V$. Define $\hat
  f=f-f(0)$. Then $\hat f$ is a step-isometry that fixes zero. By
  Corollary~\ref{c:isometry}, there is a linear isometry $Q$ of $V$
  such that $Q\circ \hat f$ is a step-isometry fixing $\Lambda$. Let
  $g=Q\circ \hat f$.

  By Lemma~\ref{l:step-isometry-fixed-lambda}, $g$ factorises over the
  $\ell_\infty$-decomposition $V=(U\oplus\ell_\infty^d)_\infty$ as $g_U\oplus
  g_\lid$ where $g_U$ is the identity on $U$, and $g_\lid$ is a
  step-isometry on $\lid$.

  Obviously $Q^{-1}$ is a linear isometry of $V$, so, by
  Corollary~\ref{c:isom-decomp}, it factorises as $q_U\oplus q_\lid$
  over $U\oplus \lid$ and is a isometry on each part.  Note that,
  $q_u$ and $q_\lid$ are both bijective (either immediate from
  linearity, or from Corollary~\ref{c:fin-dim-mazur-ulam}).

  Define $f_U=q_U\circ g_U$ and $f_\lid=q_\lid\circ g_\lid$. By
  definition $f_U$ maps $U$ to itself isometrically, and $f_\lid$ maps
  $\lid$ to itself as a step-isometry.
  Furthermore, 
  \begin{align*}
    f(u+w)&=Q^{-1}(g(u+w))=Q^{-1}(g_U(u)+g_\lid(w))\\
    &=q_u(g_U(u))+q_\lid(g_\lid(w))\\
    &=f_u(u)+f_\lid(w),
  \end{align*}
  i.e., $f=f_U\oplus f_\lid$ is a factorisation of $f$ over $V=U\oplus
  \lid$. This completes the proof.
\end{proof}

\section{Proof of Theorem~\ref{t:l_infinity}}\label{s:l-inf-isom}
In this section we use the results we have proved to deduce
Theorem~\ref{t:l_infinity}.  We prove it first for the case $d=1$:
i.e., $V=\R$.
\begin{lemma}\label{l:step-isom-R}
  Suppose $f$ is a step isometry of $\R$. Then there exists
  an isometry $Q$ of $\R$ and a continuous increasing bijection
  $g\colon[0,1)\to[0,1)$ such that
\[Q\circ f(x)=\lfloor x\rfloor +g(x-\lfloor x\rfloor).\]
\end{lemma}
\begin{proof}
  Trivially, the lattice $\Lambda$ generated by the unit ball is just
  the set $\Z$. Thus, by Corollary~\ref{c:isometry}, there exists an
  isometry $Q$ such that $Q\circ f$ fixes $\Z$ and, by
  Corollary~\ref{c:Lambda-preserved}, \[Q\circ f(x+k)=Q\circ
  f(x)+k\tag{$*$}\] for any $x\in \R$ and $k\in \Z$. Let $\hat
  f=Q\circ f$.  Since $\hat f$ is a step isometry and fixes both $0$
  and $1$, it must map $(0,1)$ to $(0,1)$, as must $\hat
  f^{-1}$. Hence, defining $g=\hat f|_{[0,1)}$ we see that $g$ maps
  $[0,1)$ to $[0,1)$ bijectively.  From $(*)$ we see that \[ Q\circ
  f(x)=\lfloor x\rfloor +g(x-\lfloor x\rfloor).\]

  It is immediate that $g$ is continuous (it a restriction of the
  continuous function $\hat f$), so, to complete the proof, we just
  need to show that $g$ is increasing.  Suppose that $0\le x<y<
  1$. Pick $z\in (1+x,1+y)$. We showed above that $\hat f$ maps
  $(0,1)$ to itself and similarly it also maps $(1,2)$ to itself; in
  particular, $\hat f(z)\in (1,2)$. Thus, since $\hat f$ is a step
  isometry, we have $\hat f(z)>1+\hat f(x)=1+g(x)$ and $\hat
  f(z)<1+\hat f(y)=1+g(y)$, which shows $g(x)<g(y)$ as claimed.
\end{proof}
\begin{proof}[Proof of Theorem~\ref{t:l_infinity}] By
  Corollary~\ref{c:isometry} there is an isometry $Q$ such that
  $Q\circ f$ is a step isometry fixing $\Lambda$ the lattice generated
  by the extreme points of $B$.

  By Proposition~\ref{p:ext-line=preserved-line}, the step isometry
  $Q\circ f$ preserves extreme line directions. It is obvious that the
  points $\sum_{i=1}^d e_i$ and $\sum_{i=1}^d e_i -2e_j$ are endpoints
  of an extreme line with direction $e_j$. Thus, each coordinate
  direction $e_j$ is preserved, and we see that $Q\circ f$ decomposes
  into independent actions on each coordinate direction. Each of these
  has the form specified by Lemma~\ref{l:step-isom-R}. Since $Q$ has
  the form given by Corollary~\ref{c:l-inf-isom} the result follows.
\end{proof}

\section{The Back and Forth Method in Our Setting}\label{s:bf}
A standard technique for proving infinite graphs are isomorphic is the
`back and forth' method.  As we shall use it several times in the
proof of Theorem~\ref{t:precise-main-theorem} from
Theorem~\ref{t:step-isometry}, we collect precisely what we need
here.

\begin{lemma}\label{l:bf}
  Let $V=(U\oplus\R)_\infty$ and let $S_U$ be a countable dense subset
  of $U$. Suppose that $S$ is a countable dense subset of $V$ such
  that, for each $s\in S_U$, $S\cap (\{s\}\times\R)$ is dense in
  $\{s\}\times\R$, and no two points in $S$ differ by an integer in
  their $\R$-component. Then $S$ is Rado.

Further suppose $S_0$ is any finite set
  of points in $V$ with no two points, one from $S$ and one from
  $S_0$, differing by an integer in their $\R$-components. Then, for
  two graphs $G,G'$ in $\cG_p(S\cup S_0)$, we have
  \[\Prb(G\equiv G'\ |\ G[S_0]=G'[S_0])=1
  \]
(where, as usual, $G[S_0]$ denotes the supgraph of $G$ restricted to $S_0$).
\end{lemma}
\begin{remark}
  Note, we do not require this to be the $\ell_\infty$-decomposition:
  for example, it also holds for
  $V=\ell_\infty^d=(\ell_\infty^{d-1}\oplus \R)_\infty$ itself. Indeed,
  we do not even need $U$ to be non-trivial: i.e., it works for
  $V=\R$.
\end{remark}
\begin{proof}
  We start by showing that almost all graphs $G$ in $\cG_p(V,S)$ have
  the following property $P$: for every point $s'\in S_U$, every open
  subset $A$ of $\R$, and every pair of disjoint finite sets
  $T_1,T_2\subset S$ such that $\{s'\}\times A\subset \bigcap_{x\in T_1\cup
    T_2}B^\circ(x,1)$, there exist infinitely many $s \in (\{s'\}\times A)\cap
  S$ such that $st\in E(G)$ for all $t\in T_1$ and $st\not\in E(G)$
  for all $t\in T_2$.

  It is obviously sufficient to prove the claim for all open sets in
  any base for $\R$. In particular, if we take a countable base, there
  are only countably many choices for $A$, $s'$, $T_1$ and $T_2$. For
  each choice there are infinitely many points in $(\{s'\}\times A)\cap
  S$. Since, each of these points has distance strictly less than one
  to each point of $T_1\cup T_2$, each such point has a positive
  probability of having the required adjacency.  Thus, almost surely,
  infinitely many of them do have the required adjacency. The claim
  follows.

  To complete the proof we show that, if $G$ and $G'$ are two graphs in
  $\cG_p(V,S)$ both having property $P$, then $G$ and $G'$ are
  isomorphic.

  Indeed, we construct our isomorphism guaranteeing that it factorises
  over $U\oplus \R$ as $f_U\oplus f_\R$ and that $f_U$ is actually the
  identity on $U$. In other words $f(u+w)=u+f_\R(w)$. Further, we
  insist that $f_\R$ is monotone and satisfies
  \[ f_\R=\lfloor x\rfloor +f_\R(x-\lfloor x\rfloor)\tag{$**$} \] (in
  fact this is essentially forced if $f_\R$ is to be a step-isometry).

  For the rest of the proof fix an enumeration $s_1,s_2,s_3,\dots$ of
  $S$. We use the back and forth method to construct the desired
  isomorphism. Start the process by mapping $s_1=u_1+w_1$ to
  itself. In particular, this defines $f_\R(w_1)=w_1$, so, by our
  requirement on $f_\R$, this defines $f_\R$ on $w_1+\Z$ by
  $f_\R(w_1+k)=w_1+k$, for all $k\in \Z$.

  Suppose that $v=u+w$ is the first point in the enumeration for which
  $f$ has not already been defined, and that $v_i=u_i+w_i$, for $1\le
  i\le n$, are the points for which $f$ has already been defined. Let
  $v_i'=f(v_i)$ for each~$i$. Consider the set of points for which we
  have already defined $f_\R$, namely $\bigcup_i (w_i+\Z)$. The
  point $w$ must lie between two consecutive of these points, say $x$
  and $y$. (It is not one of these points since we have assumed there
  are no two points differ by an integer in their $\R$-component.)

  Let $A$ be the open interval $(f_\R(x),f_\R(y))$, and let $T$ be the
  subset of the $v_i'$ that have distance strictly less than one from
  any point (equivalently, all points) of $\{u\}\times A$, and
  partition $T$ into $T_1$ and $T_2$ according to whether $v_i$ is
  joined to $v$ in $G$ or not.

  Since $G'$ has property $P$, there are infinitely many points $v'\in
  \{u\}\times A$ which are joined to everything in $T_1$ and nothing
  in $T_2$. Let $v'$ be any such point that has not already been used
  -- i.e., is not in $\{v_1',v_2',\dots,v_n'\}$ -- and set
  $f(v)=v'$. Let $w'\in \R$ be the $W$-component of $v'$ (so
  $v'=u+w'$). By our requirement on $f_\R$ this defines $f_\R$ on
  $w+\Z$ by $f_\R(w+k)=w'+k$, for all $k\in \Z$. By our choice of $v'$
  we see that $f_\R$ is still a monotone and increasing and satisfies
  $(**)$.

  We repeat this argument, but this time mapping from $G'$ to
  $G$. That is, we take the first point $v'$ in our enumeration of $S$
  that is not one of the $v_i'$ and define $f^{-1}(v')=v$ for a
  suitable point $v$ found as above but working with $f^{-1}$. 

  Thus, since as we alternate back and forth the process takes the
  first point not yet defined in $G$ or $G'$ at each stage, this
  process creates a bijection. Since we maintain the isomorphism and
  the step-isometry at each stage this bijection is an isomorphism
  (and a step isometry) as claimed.

  To prove the final part just start the process with the map $f\colon
  S_0\to S_0$ defined to be the identity which, since we are
  conditioning on $G[S_0]=G'[S_0]$, is an isomorphism.
\end{proof}

\section{Proof of Theorems~\ref{t:main-theorem}
  and~\ref{t:precise-main-theorem} from
  Theorem~\ref{t:step-isometry}}\label{s:main-theorem}

In this section we prove Theorem~\ref{t:precise-main-theorem} (which
includes Theorem~\ref{t:main-theorem}).
\begin{lemma}\label{l:countable=not-isomorphic}
  Let $V$ be a finite-dimensional normed space and $\cds$ be a
  countable dense set in $V$.

  \begin{enumerate}
    \item Suppose that there are only countably many step-isometries
    on $S$. Then $S$ is strongly non-Rado.
    \item Instead, suppose that $S$ contains a subset $T$ which
    contains infinitely many pairs of points at distance less than
    one, and the step-isometries on $S$ induce only countably many
    distinct mappings of $T$, then $S$ is strongly non-Rado.
  \end{enumerate}
\end{lemma}
\begin{proof}
Obviously the second statement gives the first statement, so it
suffices to prove that.

  Let $P$ be the property that for every pair of points $x,y\in S$ and
  every $k\in \N$ with $k\ge 2$ we have $\|x-y\|<k$ if and only
  $d_G(x,y)\le k$.  Let $\cG_0$ be the set of graphs for which
  property $P$ fails. By Lemma~\ref{l:bonato-janssen}, $\cG_0$ has
  measure zero. Any $G\not\in \cG_0$ can only be isomorphic to graphs
  in $\cG_0$ or to a graph $f(G)$ where $f$ is step-isometry of
  $S$. Obviously, if $f$ is an isomorphism between $G$ and $G'$, then
  $f|_T$ is an isomorphism between $G[T]$ and $G'[f(T)]$. Since $T$
  has infinitely many pairs of points at distance less than one, it
  has infinitely many potential edges, and the probability any
  particular mapping $f|_T$ is an isomorphism is zero. By hypothesis
  there are only countably many such mappings so the probability that
  any such mapping is an isomorphism is zero.

  To sum up, almost every $G$ is isomorphic to almost no graphs. Thus,
  by Fubini's theorem, two independent random graphs are almost surely
  not isomorphic. (The event that two graphs are isomorphic, although
  not Borel, is product measurable because it is analytic -- see
  e.g.,~\cite{MR2526093}.)
\end{proof}

Throughout the proof of Theorem~\ref{t:precise-main-theorem} we shall
use the $\ell_\infty$-decomposition. We make the following definition.
\begin{defn}
  Suppose $V$ is a normed space with $\ell_\infty$-decomposition
  $V=(U\oplus \ell^d_\infty)_\infty$. Then, for any $u\in U$, the
  \emph{fibre} over $u$ is the set $\{u+w:w\in\ell_\infty^d\}$.
\end{defn}

\begin{proof}[Proof of Theorem~\ref{t:precise-main-theorem}(\ref{e:ii})]
  Suppose $f$ is a step-isometry of $S$. By
  Proposition~\ref{p:extends}, $f$ extends to a step-isometry of
  $V$. Since $U=V$ in the $\ell_\infty$-decomposition,
  Theorem~\ref{t:step-isometry} shows that $f=f_U$ must be a
  (bijective) isometry on the whole of $V$.  By the Mazur-Ulam theorem
  this isometry is an affine map.

  Let $S'\subset S$ be an affine basis of $V$, i.e., a linear basis
  together with any one point not in its affine span.  Then, the
  affine map, $f$, is determined by its action on $S'$. Since $f$ maps
  $\cds$ to $\cds$, there are only countable many choices for the
  images of the points of $S'$. Hence, the number of such isometries
  is countable.

  This shows that the number of step-isometries on $S$ is
  countable so, by Lemma~\ref{l:countable=not-isomorphic}, $\cds$ is
  strongly non-Rado.
\end{proof}
\begin{proof}[Proof of Theorem~\ref{t:precise-main-theorem}(\ref{e:iii})]
  First suppose that no two (distinct) points $u+w, u'+w'\in S$ have
  $u=u'$ (i.e. each fibre over $U$ contains zero or one
  point). Obviously, almost all countable dense sets have this
  property. Again suppose that $f$ is a step-isometry of $S$. As
  before, it extends to a step-isometry of $V$. By
  Theorem~\ref{t:step-isometry}, $f$ factorises as $f=f_U\oplus
  f_{\ell_\infty^d}$, where $f_U$ is a (bijective) isometry
  on~$U$. Thus, by the Mazur-Ulam theorem again, $f_U$ is an affine
  map.

  Let $S'\subset S$ be a set of points $u_1+w_1,u_2+w_2,\dots,u_k+w_k$ where
  $u_i\in U$ and $w_i\in \ell_\infty^d$ for each $i$, and
  $u_1,u_2,\dots, u_k$ form an affine basis of $U$.  The map $f_U$ is
  determined by its action on $u_1,u_2,\dots,u_k$, so is determined by
  $f$'s action on $S'$. As in Part~(\ref{e:ii}), $f$ maps $\cds$ to
  $\cds$ so there are only countably many choices for the images of
  the points of $S'$. Thus the number of possible $f_U$ is countable.

  However, $f_U$ determines $f$ since, once we know the $U$-component
  of $f(s)$, the fact that $f(s)\in S$ determines the point uniquely
  (there may be no possible point but that only helps us since it
  reduces the number of potential step-isometries).  Hence, exactly as
  in the proof of Part~(\ref{e:ii}), this means there are only
  countably many such step-isometries so, again by
  Lemma~\ref{l:countable=not-isomorphic}, $S$ is strongly non-Rado.

  The fact that there are some sets $S$ that have atypical behaviour
  is immediate from Lemma~\ref{l:bf}. Indeed, write $V=(U'\oplus
  \R)_\infty$ where $U'=(U\oplus \ell_\infty^{d-1})_\infty$ then any
  $S$ of the form required by that lemma is Rado. We remark that this
  construction also works in the case $V=\ell_\infty^d$, but is not
  atypical there.

  Since our construction of sets for which the probability the graphs
  are isomorphic has probability strictly between $0$ and $1$ works
  for both Parts~(\ref{e:i}) and~(\ref{e:iii}) of the theorem we defer
  it until after our proof of Part~(\ref{e:i}).
\end{proof}

\begin{proof}[Proof of Theorem~\ref{t:precise-main-theorem}(\ref{e:i})]
  The `almost all' statement of Part~(\ref{e:i}) was proved by Bonato
  and Janssen. They showed that all countable dense sets that do not
  contain any two points differing by an integer in any coordinate are
  Rado.  (In fact, they claimed the slightly stronger result that any
  set which does not contain two points an integer distance apart is
  Rado -- but this is not true. Indeed, it is easy to construct
  counterexamples along the lines of the examples given in the next
  section.)

  The following shows that there are countable dense sets $S$ which
  are strongly non-Rado. Let $S'$ be any
  countable dense set in $\R^{d-1}$. Let $S=S'\times \Q$ in $\R^d$,
  and fix $s'\in S'$. Suppose $f$ is a step-isometry mapping on
  $S$. As usual $f$ extends to a step-isometry of $V$. Consider the
  action of $f$ on the subset $T=\{s'\}\times (\Z\cup \Z+\frac12)$ of
  the fibre $\{s'\}\times \Q$.  By Theorem~\ref{t:l_infinity} we see
  that this action is determined by the permutation $\sigma$ of the
  basis vectors, the vector $\eps$ of signs, together with the images
  $f(s',0)$ and $f(s',1/2)$. Since $f(s',0,),f(s',1/2)\in S$, there
  are only countably many choices for the step-isometry's action on
  $T$. Thus, since $T$ contains infinitely many pairs of points with
  distance less than one, Lemma~\ref{l:countable=not-isomorphic} shows
  that $S$ is strongly non-Rado.

  We deal with the case of sets where the probability that two graphs
  are isomorphic is strictly between zero and one in the following
  proposition.
\end{proof}

Finally we complete the proof of Theorem~\ref{t:precise-main-theorem}
by proving that there exist sets which are neither Rado nor strongly
non-Rado: i.e., for which the probability two graphs are isomorphic
lies strictly between zero and one.
\begin{proposition}
  Let $V=(U\oplus\R)_\infty$. Then there exist countable dense sets
  $S$ such that the probability that two random graphs taken from
  $\cG_p(V,S)$ are isomorphic lies strictly between zero and one.
\end{proposition}
\begin{remark}
  Again, we do not require this to be the $\ell_\infty$-decomposition:
  for example, it holds for
  $V=\ell_\infty^d=(\ell_\infty^{d-1}\oplus \R)_\infty$ and for $V=\R$.
\end{remark}

\begin{proof}
  The key idea is to find a set $S$ with some finite subset $S_0$ such
  that all step-isometries map $S_0$ to $S_0$. If we do this, then an
  obvious necessary condition for two graphs $G$ and $G'$ to be
  isomorphic via a step-isometry is that $G[S_0]$ is isomorphic to
  $G'[S_0]$, which is an event with probability strictly between zero
  and one, provided $S_0$ contains at least one possible edge.

  Of course, that is just a necessary condition; to find a set $S$
  with the desired property we wish to make this a sufficient
  condition for the existence of such an isomorphism.

  One natural possibility is to let $S_0$ be two points that are the
  unique pair of points at unit distance in $S$. Since step-isometries
  preserve integer distances any step-isometry must map $S_0$ to
  $S_0$. However, $S_0$ does not contain any potential edge. Instead,
  fix a unit vector $u$ and let $S_0=\{0,u,3u/2,5u/2\}$. Provided
  $0,u$ and $3u/2,5u/2$ are the only pairs of points at unit distance
  in $S$, then $S_0$ must map to itself. Moreover, $S_0$ contains a
  unique possible edge (i.e., pair at distance strictly less than
  one): that between the points $u$ and $3u/2$, and we see that any
  step-isometry must map these two points to themselves.

  Having found our set $S_0$ we turn to defining $S$, which we do as
  in Lemma~\ref{l:bf} -- we just add the requirements that no point of
  $S$ is at unit distance from any point in $S\cup S_0$.

  As discussed above all step isometries map the set $\{u,3u/2\}$ to
  itself and, in particular, a necessary condition for $G$ and $G'$
  to be isomorphic via a step-isometry is that they agree on the
  potential edge $u,3u/2$. (As Lemma~\ref{l:bonato-janssen} shows that
  the probability two graphs are isomorphic via a function which is
  not a step-isometry is zero, we can ignore this possibility.)

  Conversely, if they agree on this edge then $G[S_0]=G'[S_0]$ so, by
  Lemma~\ref{l:bf} they are almost surely isomorphic.

  Thus, the probability that $G$ and $G'$ are isomorphic is the
  probability that they agree on the edge $u,3u/2$ which is
  $p^2+(1-p)^2$; in particular it is strictly between zero and one.
\end{proof}

\section{Further Results and Open Problems}

We have not completely classified the
behaviour of all countable dense sets in the cases~(\ref{e:i})
and~(\ref{e:iii}) above, and that is our main open question
\begin{question}
  Let $V$ be a normed space with $\ell_\infty$-decomposition
  $V=(U\oplus \lid)_\infty$ for some $d\ge 1$. Which countable dense
  sets are Rado?
\end{question}
It is easy to extend the argument for the typical case of
Part~(\ref{e:iii}) above to show that, in that setting, if each fibre
over $U$ contains a discrete set (rather than just zero or one points
as above), then the set is strongly non-Rado.  Thus, the open cases
include cases where a fibre is neither dense nor discrete.

However, since the behaviour when all fibres are discrete (strongly
non-Rado) is different from the case when all fibres are dense (Rado
-- assuming some no integer difference conditions) it is unsurprising
that sets with some fibres discrete and some fibres dense can give
either behaviour. We briefly outline two sets which look very similar
but have different behaviour. The examples we give are in in
$V=(U\oplus\R)_\infty$ but it is easy to generalise them to either
$(U\oplus\lid)_\infty$ or (with slightly more effort along the lines
of the proof of the atypical case of Part~(\ref{e:i}) above) to $\lid$.

Let $S_U$ be a dense set in $U$, and let $S$ be a set which is dense
in each fibre over $S_U$, and contains no two points differing by an
integer in their $\R$-component. (So far this is exactly the set used
in the atypical case of Part~(\ref{e:iii}) above.)

Now let $T_U$ be an infinite one separated family in $U$ disjoint from
$S_U$, and let $T$ be a set containing exactly one point from each
fibre over $T_U$, such that no two points in $S\cup T$ differ by an
integer in their $\R$-component.

We claim that, by choosing the single points in each fibre of $T$, we
can ensure that $S\cup T$ is Rado, or that it is strongly
non-Rado. Suppose that $T$ is the set $\{(t_1,r_1),(t_2,r_2),\dots\}$.

As usual, any step-isometry $f$ of $S\cup T$ extends to a
step-isometry of $V$, which factorises as $f_U\oplus f_\R$ where $f_U$
is an isometry and $f_\R$ is a step-isometry. Obviously $f_U$ maps $T$
to itself (as all other fibres contain either no points or infinitely
many points). Thus, once we know the $U$-component of the image $f(t)$
of a point $t\in T$, we know its $\R$-component; i.e., $f_U$
determines $f_\R(r_i)$ for each $i$. If the $r_i \mod 1$ are dense in
$[0,1]$ then, since $f_\R$ is a step-isometry this determines $f_\R$
entirely. As in our proofs above there are only countably many
step-isometries mapping $S\cup T$ to itself so, by
Lemma~\ref{l:countable=not-isomorphic}, $S\cup T$ is strongly
non-Rado.

On the other hand if the $r_n=n+1/n$ and no point of $S$ has integer
$\R$-component then $S\cup T$ is Rado. Indeed, we construct our map
fixing $U$ and use the `back and forth' argument as in Lemma~\ref{l:bf}
observing that the key property used there -- that for every point
$(u,w)\in S\cup T$ not yet mapped the point $w$ lies in an open
interval between consecutive previously defined points -- still holds
in this case.

The above discussion shows that the classification of exactly which
countable dense sets give a unique graph will be rather
complicated. Thus we have restricted ourselves to the `typical' case
and showing that the atypical cases can occur.

Finally, all our work in this paper has been finite-dimensional spaces
without consideration for the infinite-dimensional setting. It would
be interesting to know what happens there.
\begin{question}
  Suppose that $V$ is an infinite-dimensional normed space, and that
  $S$ is a countable dense subset. When is $S$ Rado?
\end{question}
 \bibliography{mybib}{}
\bibliographystyle{abbrv}
\end{document}